\documentclass[reqno,10pt]{amsart}
\numberwithin{equation}{section}
\usepackage{latexsym}
\usepackage{amsmath}
\usepackage{amssymb}
\usepackage{mathrsfs}
\usepackage{graphicx,colordvi}
\usepackage{upgreek}
\usepackage{ifthen}
\usepackage{color}

\usepackage[T1]{fontenc}
\usepackage[latin1]{inputenc}

\numberwithin{equation}{section}

\newtheorem{defi}{Definition}[section]
\newtheorem{theorem}[defi]{Theorem}
\newtheorem{lemma}[defi]{Lemma}
\newtheorem{corollary}[defi]{Corollary}
\newtheorem{proposition}[defi]{Proposition}
\newtheorem{remark}[defi]{Remark}
\newtheorem{remarks}[defi]{Remarks}
\newtheorem{example}[defi]{Example}

\usepackage{latexsym}
\usepackage{amsmath}
\usepackage{amssymb}

\newcommand{\cD}{{\mathcal D}}

\newcommand{\cE}{{\mathcal E}}

\newcommand{\cP}{{\mathcal P}}
\newcommand{\cT}{{\mathcal T}}

\newcommand{\R}{{\mathbb R}}

\renewcommand{\epsilon}{\varepsilon}

\newcommand{\bb}{{\bb}}

\newcommand{\nn}{\nonumber}
\begin{document}

\title[On the Navier-Stokes equations on surfaces]{On the Navier-Stokes equations on surfaces}

\author{Jan Pr\"uss$\dag$ }

\author{Gieri Simonett}
\address{Department of Mathematics\\
        Vanderbilt University\\
        Nashville, Tennessee\\
        USA}
\email{gieri.simonett@vanderbilt.edu}

\author{Mathias Wilke}
\address{Martin-Luther-Universit\"at Halle-Witten\-berg\\
         Institut f\"ur Mathematik \\
         Halle (Saale), Germany}
\email{mathias.wilke@mathematik.uni-halle.de}

\thanks{This work was supported by a grant from the Simons Foundation (\#426729, Gieri Simonett).}



\dedicatory{Dedicated to Matthias Hieber on the occasion of his $60^{th}$ birthday}

\begin{abstract}
We consider the motion of an incompressible viscous 
fluid that completely covers a smooth, compact and embedded hypersurface $\Sigma$  without boundary  and flows along $\Sigma$. Local-in-time well-posedness is established in the framework of 
 $L_p$-$L_q$-maximal regularity. We characterize the set of equilibria as the set of all Killing vector fields on $\Sigma$ and we show that each equilibrium on $\Sigma$ is stable.  Moreover, it is shown that any solution starting close to an equilibrium exists globally and converges at an exponential rate to a (possibly different) equilibrium as time tends to infinity.
\end{abstract}

\maketitle

\section{Introduction}
Suppose $\Sigma$ is a smooth, compact, {connected}, embedded (oriented) hypersurface in $\R^{d+1}$ without boundary.
Then we consider the motion of an incompressible viscous fluid that completely covers $\Sigma$
and flows along $\Sigma$.

Fluid equations on manifolds appear in the literature as mathematical models for various physical and biological processes,
for instance in the modeling of emulsions and biological membranes.
The reader may also think of an aquaplanet whose surface is completely covered by a fluid.
The case of a planet with oceans and landmass will be considered in future work.

Fluid equations on manifolds have also been studied as mathematical problems in their own right,
see for instance
\cite{ArCr12, CCD17, EbMa70, JaOlRe17, KLG17, Maz03, Tay92, Tem88} and the references cited therein.

In this paper, we model the fluid by the `surface Navier-Stokes equations' on $\Sigma$,
using as constitutive law the {\em Boussinesq-Scriven surface stress tensor}
\begin{equation}\label{T-Sigma-0}
\cT_\Sigma =\cT_\Sigma(u,\pi)=2\mu_s \cD_\Sigma(u) + (\lambda_s-\mu_s)({\rm div}_\Sigma u) \cP_\Sigma -\pi \cP_\Sigma,
\end{equation}
where $\mu_s$ is the surface shear viscosity,
$\lambda_s$ the surface dilatational viscosity, $u$ the velocity field, $\pi$  the pressure, and
\begin{equation}\label{D-Sigma}
\cD_\Sigma(u):=\frac{1}{2}\cP_\Sigma\left( \nabla_\Sigma u + [\nabla_\Sigma u]^{\sf T}\right)\cP_\Sigma
\end{equation}
the surface rate-of-strain tensor.
Here,
$\cP_\Sigma$ denotes the orthogonal projection onto the tangent bundle ${\sf T}\Sigma$ of $\Sigma$,
${\rm div}_\Sigma$  the surface divergence, and $\nabla_\Sigma$ the surface gradient.
We refer to Pr\"uss and Simonett~\cite{PrSi16} and the Appendix for more background information on these objects.

Boussinesq~\cite{Bou13} first suggested to consider
surface viscosity to account for intrinsic frictional forces within an interface.
Several decades later, Scriven~\cite{Scr60} generalized Boussinesq's approach
to material surfaces having arbitrary curvature.
The resulting tensor is nowadays called the Boussinesq-Scriven stress tensor.

For an incompressible fluid, i.e., ${\rm div}_\Sigma u=0$,  the Boussinesq-Scriven surface tensor simplifies to
\begin{equation}\label{T-Sigma}
\cT_\Sigma =2\mu_s \cD_\Sigma(u) -\pi \cP_\Sigma.
\end{equation}

We then consider the following {\em surface Navier-Stokes equations} for an incompressible viscous fluid
    \begin{equation}
    \label{NS-surface}
    \begin{aligned}
    \varrho \big(\partial_t u + \cP_\Sigma (u\cdot \nabla_\Sigma u)\big) - \cP_\Sigma\, {\rm div}_\Sigma\, \cT_\Sigma &=0 &&\text{on}\;\;\Sigma \\
    {\rm div}_\Sigma u &=0 &&\text{on} \;\; \Sigma \\
                           u(0) &= u_0  &&\text{on}\;\; \Sigma,
    \end{aligned}
    \end{equation}
where $\varrho$ is a positive constant. In the sequel, we will always assume that $u_0\in {\sf T}\Sigma$,
i.e., $u_0$ is a tangential field.
\begin{remark}\label{rem:tangential}
{\rm
Suppose $u_0\in {\sf T}\Sigma$.
If  $(u(t),\pi(t))$ is a (sufficiently) smooth solution to \eqref{NS-surface} on some time interval $[0,T)$
then we also have $u(t)\in {\sf T}\Sigma$  for  all $t\in [0,T]$.
This can readily be seen by taking the inner product of the first equation in~\eqref{NS-surface}
with $\nu_\Sigma(p)$, yielding $(\partial_t u(t, p) | \nu_\Sigma(p))=0$ for $(t,p)\in [0,T)\times\Sigma$,
where $\nu_\Sigma$ is the unit normal field of $\Sigma$.
Hence, $(u(t,p)|\nu_\Sigma(p))=(u_0(p) | \nu_\Sigma(p))=0$ for $(t,p)\in [0,T)\times \Sigma.$
}\end{remark}
\noindent
It will be shown in the Appendix that \eqref{NS-surface} can we written in the form
\begin{equation}
\label{NS-surface-2}
\begin{aligned}
\varrho \big(\partial_t u + \cP_\Sigma (u\cdot \nabla_\Sigma u)\big)  - \mu_s \Delta_\Sigma u
- \mu_s(\kappa_\Sigma L_\Sigma - L^2_\Sigma)u + \nabla_\Sigma \pi&=0 &&\text{on}\;\;\Sigma \\
{\rm div}_\Sigma u &=0 &&\text{on} \;\; \Sigma \\
                       u(0) &= u_0  &&\text{on}\;\; \Sigma,
\end{aligned}
\end{equation}
where $\Delta_\Sigma$ is the (negative) Bochner-Laplacian, $\kappa_\Sigma$ the $d$-fold 
mean curvature of $\Sigma$
(the sum of the principal curvatures), and $L_\Sigma$ the Weingarten map.
Here we use the convention that a sphere has negative mean curvature.

We would like to emphasize that the tensor
 $(\kappa_\Sigma L_\Sigma - L^2_\Sigma)$ is  an intrinsic quantity.
In fact, we shall show in Proposition~\ref{pro:D-Sigma}  that
\begin{equation}
\label{Ricci-Gauss}
\begin{aligned}
& \kappa_\Sigma L_\Sigma - L^2_\Sigma = {\rm Ric}_\Sigma  \\
& \kappa_\Sigma L_\Sigma - L^2_\Sigma  = K_\Sigma\quad \text{in case $d=2$,}
\end{aligned}
\end{equation}
where  ${\rm Ric}_\Sigma$ is the Ricci tensor and
 $K_\Sigma$  the Gaussian curvature of $\Sigma$ (the product of the principal curvatures). 
 
 \medskip

 The formulation \eqref {NS-surface} coincides with \cite[formula (3.2)]{JaOlRe17}.
In that paper, the equations for the motion of a viscous incompressible fluid
on a surface were derived from fundamental continuum mechanical principles.
The same equations were also derived in~\cite[formula (4.4)]{KLG17},
based on global energy principles.
We mention that the authors of  \cite{JaOlRe17, KLG17} also consider material surfaces that
may evolve in time.

Here we would like to point out that several formulations for the
`surface Navier-Stokes equations' have been used in the literature,
see~\cite{CCD17} for a comprehensive discussion, and also~\cite[Section 3.2]{JaOlRe17}.
It turns out that the model based on the {Boussinesq-Scriven surface stress tensor}
leads to the same equations as, for instance, 
in \cite[Note added to Proof]{EbMa70} and~\cite{Tay92}.
Indeed, this follows from \eqref{NS-surface-2}-\eqref{Ricci-Gauss} and the relation
\begin{equation}\label{Hodge}
\Delta_\Sigma u= \Delta_H u+ {\rm Ric_\Sigma}u,
\end{equation}
where $\Delta_H$ denotes the Hodge Laplacian (acting on $1$-forms).
\goodbreak

\medskip
The plan of this paper is as follows.
In Section 2 we show that kinetic energy is dissipated by the fluid system~\eqref{NS-surface},
and we characterize all the equilibrium solutions of~\eqref{NS-surface}.
It is shown that at equilibrium, the gradient of the  pressure is completely determined by the velocity field.
Moreover, it is shown that the equilibrium (that is, the stationary) velocity fields 
correspond exactly to the Killing fields of $\Sigma$.
We finish Section 2 with some observations concerning the motion of fluid particles 
in the case of a stationary velocity field.

In Section 3, we prove that the linearization of \eqref{NS-surface-2} enjoys the property
of $L_p$-$L_q$-maximal regularity. We rely on results contained in~\cite[Sections 6 and 7]{PrSi16}.
Moreover, we introduce the Helmholtz projection on $\Sigma$ and we prove
interpolation results for divergence-free vector fields on $\Sigma$.
We then establish local well-posedness of \eqref{NS-surface-2} in 
the (weighted) class of $L_p$-$L_q$-maximal regularity, see Theorem~\ref{Theorem-LWP}.

In Section 4, we prove that all equilibria of \eqref{NS-surface-2} are stable.
Moreover, we show that any solution starting close to an equilibrium exists globally and converges at an exponential rate to a (possibly different) equilibrium as time tends to infinity.
In order to prove this result we show that each equilibrium is normally stable. 
Let us recall that the set of equilibria $\cE$ coincides with the vector space of all
Killing fields on $\Sigma$.
It then becomes an interesting question to know how many Killing fields a given manifold can support.
In Subsection~\ref{Killing}, we include some remarks about the dimension of $\cE$ and we discuss some examples.

In forthcoming work, we plan to use  the techniques introduced in this manuscript
to study the Navier-Stokes equations on manifolds with boundary.

\medskip
We would like to briefly compare the results of this paper with previous results by other authors.
Existence of solutions for the Navier-Stokes equations~\eqref{NS-surface-2}
has already been established in \cite{Tay92}, see also~\cite{Maz03} 
and the comprehensive list of references in~\cite{CCD17}.
The authors in~\cite{Tay92, Maz03} employ techniques of pseudo-differential operators and they make
use of the property that the Hodge Laplacian commutes with the Helmholtz projection.
Under the assumption that the spectrum of the linearization is contained in the negative real axis,
stability of the zero solution is shown in \cite{Tay92}.
The author remarks that this assumption implies that the isometry group of $\Sigma$ is discrete.
In contrast, our stability result in Theorem~\ref{thm:stability} applies to any manifold.

The Boussinesq-Scriven surface stress tensor has also been employed
in the situation of two incompressible fluids which are separated by a free surface, where
surface viscosity (accounting for internal friction within the interface) is included in the model,
see \cite{BoPr10}.

Finally, we mention~\cite{JaOlRe17, OQRY18, ReZh13,ReVo18} and the references 
contained therein for interesting numerical investigations.
These authors also observed that the equilibria velocities correspond to Killing fields.
\goodbreak

\section{Energy dissipation and equilibria}\label{sec:-energy}
\noindent
In the following, we set $\varrho=1$.
Let
\begin{equation}\label{energy-def}
{\sf E}(t):=\int_\Sigma \frac{1}{2} |u(t)|^2\,d\Sigma
\end{equation}
be the (kinetic) energy of the fluid system.
We show that the energy is dissipated by the fluid system~\eqref{NS-surface}.
\begin{proposition}\label{pro:energy-dissipation}
Suppose $(u,\pi)$ is a sufficiently smooth solution of \eqref{NS-surface} with initial value
$u_0\in {\sf T}\Sigma$, defined on some interval  $(0,T)$.  Then
\begin{equation}\label{energy-dissipation}
\frac{d}{dt}{\sf E}(t) = -2\mu_s \int_\Sigma |\cD_\Sigma(u(t))|^2\,d\Sigma,
\quad t\in (0,T).
\end{equation}
\end{proposition}
\begin{proof}
By Remark~\ref{rem:tangential} we know that  $u(t)\in {\sf T}\Sigma$ for each $t \in (0,T)$.
In order not to overburden the notation, we suppress the variables $(t,p)\in (0,T)\times\Sigma $ in the
following computation.
It follows from \eqref{divergence-parts} and Lemma~\ref{lem:app} that
\begin{equation*}
\begin{aligned}
\big(\cP_\Sigma(u\cdot \nabla_\Sigma u)\big| u) &=\frac{1}{2} \big(\nabla_\Sigma |u|^2 \big| u\big), \quad
\big(\cP_\Sigma {\rm div}_\Sigma (\pi \cP_\Sigma) \big| u)=(\nabla_\Sigma \pi |u), \\
\big( { \rm div}_\Sigma \cD_\Sigma (u)\big| u)&= {\rm div}_\Sigma (\cD_\Sigma(u)u)- |\cD_\Sigma(u)|^2,
\end{aligned}
\end{equation*}
where $(\cdot | \cdot)$ denotes the Euclidean inner product.
Hence, the surface divergence theorem~\eqref{divergence-theorem}
and the relation  ${\rm div}_\Sigma u=0$ yield
\begin{equation*}
\begin{aligned}
\frac{d}{dt} {\sf E}(t)
&= \int_\Sigma (\partial_t u | u)\,d\Sigma  \\
&= \int_\Sigma \big(\! - \! \big(P_\Sigma(u \cdot \nabla_\Sigma u) \big| u\big)
+2\mu_s\big(P_\Sigma{\rm div}_\Sigma \cD_\Sigma(u) \big| u\big) -
 \big(P_\Sigma{\rm div}_\Sigma (\pi\cP_\Sigma) \big| u\big) \big)\,d\Sigma\\
&= -2\mu_s \int_\Sigma |\cD_\Sigma(u)|^2.
\end{aligned}
\end{equation*}
\end{proof}
\noindent
We now characterize the equilibria of  \eqref{NS-surface}.
It will turn out that at equilibrium, the gradient of the  pressure is completely determined by the velocity.
Moreover, the equilibrium velocity fields correspond exactly to the Killing fields of $\Sigma$.
\begin{proposition}\label{pro:equilibria}
\mbox{}
\begin{itemize}
\item[(a)]
Let ${\mathfrak E}:=\{(u,\pi)\in C^2(\Sigma, {\sf T}\Sigma)\times C^1(\Sigma): (u,\pi) \text{ is an equilibrium
for~\eqref{NS-surface}}\}$. Then 
$$ {\mathfrak E}=\Big\{(u,\pi): {\rm div}_\Sigma u=0,\; \cD_\Sigma(u)=0,\; \pi=\frac{1}{2}|u|^2+c\Big\},$$
where $c$ is an arbitrary constant.
\vspace{2mm}
\item[(b)] 

Suppose $\cD_\Sigma(u)=0$ for $u\in C^1(\Sigma, {\sf T}\Sigma)$.
Then ${\rm div}_\Sigma u=0$.
\vspace{2mm}
\item[(c)]
The $C^1$-tangential fields satisfying the relation $\cD_\Sigma(u)=0$ correspond exactly to
the {Killing fields} of $\Sigma$.
\end{itemize}
\end{proposition}
\begin{proof}
(a) Suppose $(u,\pi)$ is an equilibrium of \eqref{NS-surface}.
By the second line in \eqref{NS-surface}, ${\rm div}_\Sigma u=0$.
It follows from Proposition~\ref{pro:energy-dissipation}
that $\cD_\Sigma(u)=0$
and the first line in~\eqref{NS-surface} then implies
 $$\cP_\Sigma( u\cdot \nabla_\Sigma u) +\nabla_\Sigma \pi=0.$$
This, together with $\cD_\Sigma(u)=0$ and  Lemma~\ref{lem:app}(b),(c), yields
\begin{equation*}
\nabla_\Sigma \pi = -\cP_\Sigma(u\cdot\nabla_\Sigma u)
=-\cP_\Sigma((\nabla_\Sigma u)^{\sf T}u) =\cP_\Sigma((\nabla_\Sigma u)u)=\frac{1}{2}\nabla_\Sigma |u|^2.
\end{equation*}
Analogous arguments show that the inverse implication also holds true.
\smallskip\\
(b) This is a direct consequence of (\ref{divergence-parts})$_3$.
\smallskip\\
(c) This follows from Remark~\ref{rem:app}(e).
\end{proof}
\medskip

Suppose that $(u,\pi)$ is an equilibrium solution of \eqref{NS-surface}.
Then we point out the following interesting observation.

Let $\gamma(s)$ be the trajectory of a fluid particle on $\Sigma$. Then $\gamma$ satisfies the
differential equation
\begin{equation*}
\dot{\gamma}(s) = u(\gamma(s)),\quad s\in\R,\quad \gamma(0)=\gamma_0\in\Sigma.
\end{equation*}
Using the assertion in Proposition~\ref{pro:equilibria}(a), we obtain
\begin{equation*}
\begin{aligned}
\frac{d}{ds} (\pi\circ \gamma) (s)
& = \big( \nabla_\Sigma \pi  \big| u\big)(\gamma(s))
 = \frac{1}{2} \big( \nabla_\Sigma |u|^2  \big | u\big) (\gamma(s)) \\
& = \big( (\nabla_\Sigma u)^{\sf T} u \big| u \big) (\gamma(s))
 = \big( \cD_\Sigma (u) u \big| u \big)(\gamma (s))=0,
\end{aligned}
\end{equation*}
as $\cD_\Sigma(u)=0$.
Hence, $\pi$ is constant along stream lines of the flow.

Furthermore,
\begin{equation*}
\begin{aligned}
\ddot{\gamma}(s) &= u(\gamma(s))\cdot\nabla_\Sigma u(\gamma(s))=u^i(\gamma(s))(\partial_i u)(\gamma(s))\\
&=u^i(\gamma(s))(\partial_i u_j-\Lambda_{ij}^ku_k)(\gamma(s))\tau^j(\gamma(s))+(l_i^j u^iu_j)(\gamma(s))\nu_\Sigma(\gamma(s))\\
&=[P_\Sigma(u\cdot\nabla_\Sigma u)](\gamma(s))+[(L_\Sigma u|u)\nu_\Sigma](\gamma(s)).
\end{aligned}
\end{equation*}
A short computation shows that
$$(\nabla_\Sigma P_\Sigma) u=l_i^j(\tau_j\otimes \nu_\Sigma)u^i,$$
for tangential vector fields $u$, hence $(u|\nabla_\Sigma P_\Sigma u)=(L_\Sigma u|u)\nu_\Sigma$.
Therefore, we obtain the relation
\begin{align}\label{constraint1}
\begin{split}
\ddot{\gamma}(s)&=[P_\Sigma(u\cdot\nabla_\Sigma u)](\gamma(s))+[(u|\nabla_\Sigma P_\Sigma u)](\gamma(s))\\
&=-[\nabla_\Sigma\pi](\gamma(s))+[(u|\nabla_\Sigma P_\Sigma u)](\gamma(s)),
\end{split}
\end{align}
since $\nabla_\Sigma\pi =-P_\Sigma(u\cdot\nabla_\Sigma u)$ in an equilibrium. Let us compare this ODE with the following second order system with constraints:
$$
\ddot{x}=f(x,\dot{x}),\quad g(x)=0.
$$
Here, $f:\R^{d}\to\R^{d}$ and $g:\R^d\to\R^{d-m}$ are smooth with ${\rm rank}\ g'(x)=d-m$ for each $x\in g^{-1}(0)$. In general, the ODE $\ddot{x}=f(x,\dot{x})$ does not leave $\Sigma:=g^{-1}(0)$ invariant. However, it can be shown that the movement of a particle under the constraint $g(x)=0$ in the force field $f$ is governed by the ODE
\begin{equation}\label{constraint2}
\ddot{x}=P_\Sigma(x) f(x,\dot{x})+(\dot{x}|\nabla_\Sigma P_\Sigma(x) \dot{x}).
\end{equation}
In other words, the effective force field is the sum of the tangential part of $f$ on $\Sigma$ and the constraint force $(\dot{x}|\nabla_\Sigma P_\Sigma(x) \dot{x})$, which results from the geodesic flow, see also \cite[Section 13.5]{PrWi19}.

Observe that the structure of \eqref{constraint1} and \eqref{constraint2} are the same.
Of course, in our situation, it follows from the ODE $\dot{\gamma}=u(\gamma)$ that $\gamma(s)\in\Sigma$, $\dot{\gamma}(s)\in \mathsf{T}_{\gamma(s)}\Sigma$, provided $\gamma(0)=\gamma_0\in\Sigma$, since $u$ is a tangential vector field on $\Sigma$. 

Finally, note that the energy $E(s):=\frac{1}{2}|\dot{\gamma}(s)|^2+\pi(\gamma(s))$ is conserved, i.e.\ $\dot{E}(s)=0$, since $(u|\nabla_\Sigma P_\Sigma u)$ is perpendicular to $\mathsf{T}\Sigma$ and $\dot{\gamma}\in\mathsf{T}\Sigma$. This known as Bernoulli's principle.

\goodbreak

\section{Existence of solutions}
In this section, we show that there exists a unique solution
$$u\in H_{p,\mu}^1((0,a);L_q(\Sigma,\mathsf{T}\Sigma))\cap L_{p,\mu}((0,a);H_q^2(\Sigma,\mathsf{T}\Sigma)),\quad \pi\in L_{p,\mu}((0,a);\dot{H}_q^1(\Sigma))$$
of \eqref{NS-surface} resp.\ \eqref{NS-surface-2} for some suitable number $a>0$. To this end, we first consider the principal linearization of \eqref{NS-surface-2} and show that the corresponding linear operator has $L_p$-$L_q$-maximal regularity in suitable function spaces. This will enable us to apply the contraction mapping principle to prove the existence and uniqueness of a strong solution to \eqref{NS-surface-2}.

\subsection{The principal linearization}

We consider the following linear problem
\begin{equation}
\label{Stokes-surface-1}
\begin{aligned}
\partial_t u+\omega u - \mu_s \Delta_\Sigma u + \nabla_\Sigma \pi&=f &&\text{on}\;\;\Sigma \\
{\rm div}_\Sigma u &=g &&\text{on} \;\; \Sigma \\
                       u(0) &= u_0  &&\text{on}\;\; \Sigma,
\end{aligned}
\end{equation}
where $\omega>0$. Here and in the sequel, we assume without loss of generality that $\varrho=1$. The main result of this section reads as follows.
\begin{theorem}\label{Lin-Theorem-1}
Suppose $\Sigma$ is a smooth, compact, connected, embedded (oriented) hypersurface in $\R^{d+1}$ without boundary and let $1<p,q<\infty$, $\mu\in (1/p,1]$. Then, there exists $\omega_0>0$ such that for each $\omega>\omega_0$, problem \eqref{Stokes-surface-1} admits a unique solution
$$u\in H_{p,\mu}^1(\R_+;L_q(\Sigma,\mathsf{T}\Sigma))\cap L_{p,\mu}(\R_+;H_q^2(\Sigma,\mathsf{T}\Sigma))=:\mathbb{E}_{1,\mu}(\Sigma),\ \pi\in L_{p,\mu}(\R_+;\dot{H}_q^1(\Sigma)),$$
if and only if the data $(f,g,u_0)$ are subject to the following conditions
\begin{enumerate}
\item $f\in L_{p,\mu}(\R_+;L_q(\Sigma,\mathsf{T}\Sigma))=:\mathbb{E}_{0,\mu}(\Sigma)$
\item $g\in L_{p,\mu}(\R_+;H_q^1(\Sigma))$, $g\in H_{p,\mu}^1(\R_+;\dot{H}_q^{-1}(\Sigma))$
\item $u_0\in B_{qp}^{2\mu-2/p}(\Sigma,\mathsf{T}\Sigma)$
\item $g(0)={\rm div}_\Sigma u_0$.
\end{enumerate}
Moreover, the solution $(u,\pi)$ depends continuously on the given data $(f,g,u_0)$ in the corresponding spaces.
\end{theorem}
\begin{remark}\mbox{}
\begin{enumerate}
\item[(a)] In Theorem \ref{Lin-Theorem-1}, we use the notations
$$\dot{H}_{q}^1(\Sigma):=\{w\in L_{1,loc}(\Sigma):\nabla_\Sigma w\in L_{q}(\Sigma,\mathsf{T}\Sigma)\},\quad\dot{H}_q^{-1}(\Sigma):=(\dot{H}_{q'}^1(\Sigma))^*$$
and we identify $g$ with the functional $[\phi\mapsto\int_\Sigma g\phi \,d\Sigma]$ on $\dot{H}_{q'}^1(\Sigma)$. 
\item[(b)] Note that the assumption $g\in \dot{H}_q^{-1}(\Sigma)$ includes the condition $\int_\Sigma g\,d\Sigma=0$.
\item[(c)] The assertion $\pi\in L_{p,\mu}(\R_+;\dot{H}_q^1(\Sigma))$ means that $\pi$ is unique up to a constant.
\item[(d)] Necessity of the conditions (1)-(4) in Theorem \ref{Lin-Theorem-1} is well known, we refer the reader e.g.\ to the monograph \cite[Chapter 7]{PrSi16}.
\end{enumerate}
\end{remark}

\subsection{Pressure Regularity}

It is a remarkable fact that the pressure $\pi$ has additional time-regularity in some special cases.
\begin{proposition}\label{Proposition-Reg-Pressure}
In the situation of Theorem \ref{Lin-Theorem-1}, assume further
$$u_0=0,\quad g=0,\quad{\rm div}_\Sigma f=0\quad\text{on}\ \Sigma.$$
Then $P_0\pi\in{_0}H_{p,\mu}^\alpha(\R_+;L_q(\Sigma))$, for $\alpha\in (0,1/2]$, where
$$P_0v:=v-\frac{1}{|\Sigma|}\int_\Sigma v \,d\Sigma$$
for $v\in L_1(\Sigma)$. Furthermore, there exists a constant $C>0$ such that the estimate
$$|P_0\pi|_{L_{p,\mu}(L_q(\Sigma))}\le C|u|_{L_{p,\mu}(H_q^1(\Sigma))}$$
is valid.
\end{proposition}
\begin{proof}
Let $\phi\in L_{q'}(\Sigma)$, $1/q+1/q'=1$ and solve the equation
$$\Delta_{\Sigma}^L\psi=P_0\phi\quad\text{on}\ \Sigma.$$
Here, $\Delta_{\Sigma}^L$ denotes the (scalar) Laplace-Beltrami operator on $\Sigma$. This yields a unique solution $\psi\in H_{q'}^2(\Sigma)$ with
$$|\nabla_\Sigma \psi|_{L_{q'}(\Sigma)} + |\nabla^2_\Sigma \psi |_{L_{q'}(\Sigma)}\le C |\phi |_{L_{q'}(\Sigma)}.$$
This follows, for instance, from  \cite[Theorem 6.4.3 (i)]{PrSi16} and the fact
$0$ is in the resolvent set of  $\Delta_{\Sigma}^L$, acting on functions with zero average. We then obtain from the surface divergence theorem \eqref{divergence-theorem}, 
\eqref{divergence-parts}, Proposition \ref{pro:D-Sigma} (a), and the fact that $({\rm div}_\Sigma f,g)=0$
\begin{align*}
  (\pi|P_0\phi)_\Sigma &= (\pi|\Delta_{\Sigma}^L\psi)_\Sigma=(\pi|{\rm div}_\Sigma(\nabla_\Sigma\psi))_\Sigma=-(\nabla_\Sigma\pi|\nabla_\Sigma\psi)_\Sigma \\
  &=(\partial_t u+\omega u-f|\nabla_\Sigma\psi)_\Sigma-\mu_s(\Delta_\Sigma u|\nabla_\Sigma\psi)_\Sigma\\
  &=\mu_s((\kappa_\Sigma L_\Sigma-L_\Sigma^2)u|\nabla_\Sigma\psi)_\Sigma-2\mu_s(P_\Sigma{\rm div}_\Sigma \mathcal{D}_\Sigma(u)|\nabla_\Sigma\psi)_\Sigma\\
  &=\mu_s((\kappa_\Sigma L_\Sigma-L_\Sigma^2)u|\nabla_\Sigma\psi)_\Sigma+2\mu_s\int_\Sigma\mathcal{D}_\Sigma(u):\nabla_\Sigma^2\psi \,d\Sigma,
\end{align*}
where $(\cdot|\cdot)_\Sigma$ denotes the inner product in $L_2(\Sigma)$ or $L_2(\Sigma,\mathsf{T}\Sigma)$.

Noting  that $\cD_\Sigma(u)\in {_0}H^{1/2}_{p,\mu}(\R_+; L_q(\Sigma))$, we may apply the fractional time-derivative $\partial_t^\alpha$ to the result
$$(\partial_t^\alpha\pi|P_0\phi)_\Sigma=\mu_s((\kappa_\Sigma L_\Sigma-L_\Sigma^2)\partial_t^\alpha u|\nabla_\Sigma\psi)_\Sigma+2\mu_s\int_\Sigma\partial_t^\alpha\mathcal{D}_\Sigma(u):\nabla_\Sigma^2\psi \,d\Sigma,$$
since $\partial_t^\alpha$ and $L_\Sigma$ commute. This yields the claim.
\end{proof}
Without loss of generality, we may always assume that $({\rm div}_\Sigma f,g,u_0)=0$. To see this, let $(u,\pi)$ be a solution of \eqref{Stokes-surface-1} and solve the parabolic problem
\begin{equation}
\label{Parabolic-surface}
\begin{aligned}
\partial_t v+\omega v - \mu_s \Delta_\Sigma v&=f &&\text{on}\;\;\Sigma \\
                      v(0) &= u_0  &&\text{on}\;\; \Sigma,
\end{aligned}
\end{equation}
by \cite[Theorem 6.4.3 (ii)]{PrSi16} to obtain a unique solution $v\in \mathbb{E}_{1,\mu}(\Sigma)$. Next, we solve $\Delta_{\Sigma}^L\Phi={\rm div}_{\Sigma} v-g$ in $\dot{
H}_q^{-1}(\Sigma)$ by \cite[Theorem 6.4.3 (i)]{PrSi16} to obtain a solution $\Phi$ such that $\nabla_\Sigma\Phi$ is unique with regularity
$$\nabla_\Sigma\Phi\in{_0}H_{p,\mu}^1(\R_+;L_q(\Sigma,\mathsf{T}\Sigma))\cap L_{p,\mu}(\R_+;H_q^2(\Sigma,\mathsf{T}\Sigma)).$$
Note that $\nabla_\Sigma\Phi(0)=0$ by the compatibility condition ${\rm div}_\Sigma u_0=g(0)$.
Define
$$\tilde{u}=u-v+\nabla_\Sigma\Phi\quad\text{and}\quad \tilde{\pi}=\pi-(\partial_t+\omega)\Phi+\psi,$$
where $\nabla_\Sigma\psi\in L_q(\Sigma,\mathsf{T}\Sigma)$ is the unique solution of
$$(\nabla_\Sigma\psi|\nabla_\Sigma\phi)_\Sigma=\mu_s(\Delta_\Sigma\nabla_\Sigma\Phi|\nabla_\Sigma\phi)_\Sigma,\quad \phi\in\dot{H}_{q'}^1(\Sigma).$$
Then $(\tilde{u},\tilde{\pi})$ solves \eqref{Stokes-surface-1} with $({\rm div}_\Sigma f,g,u_0)=0$.

Of course, the converse is also true. If $(\tilde{u},\tilde{\pi})$ solves \eqref{Stokes-surface-1} with $({\rm div}_\Sigma f,g,u_0)=0$, then one may construct a solution $(u,\pi)$ of \eqref{Stokes-surface-1} with prescribed data $(f,g,u_0)$ being subject to the conditions in Theorem \ref{Lin-Theorem-1}, by reversing the above procedure.

\subsection{Localization }

In this subsection we prove the existence and uniqueness of a solution to \eqref{Stokes-surface-1}. We start with the proof of uniqueness. To this end, let $(u,\pi)$ be a solution of \eqref{Stokes-surface-1} with $({\rm div}_\Sigma f,g,u_0)=0$.

By compactness of $\Sigma$, there exists a family of charts $\{(U_k,\varphi_k):k\in\{1,\ldots,N\}\}$ such that $\{U_k\}_{k=1}^N$ is an open covering of $\Sigma$. Let $\{\psi_k\}_{k=1}^N\subset C^\infty(\Sigma)$ be a partition of unity subordinate to the open covering $\{U_k\}_{k=1}^N$. Note that without loss of generality, we may assume that $\varphi_k(U_k)=B_{\mathbb{R}^d}(0,r)$. We call $\{(U_k,\varphi_k,\psi_k):k\in\{1,\ldots,N\}\}$ a \emph{localization system} for $\Sigma$.

Let $\{\tau_{(k)j}(p)\}_{j=1}^d$ denote a local basis of the tangent space $\mathsf{T}_p\Sigma$ of $\Sigma$ at $p\in U_k$ and denote by $\{\tau_{(k)}^j(p)\}_{j=1}^d$ the corresponding dual basis of the cotangent space ${\sf T}^*_p\Sigma$ at $p\in U_k$. Accordingly, we define $g_{(k)}^{ij}=(\tau_{(k)}^i|\tau_{(k)}^j)$ and $g_{(k)ij}$ is defined in a very similar way, see also the Appendix. Then, with $\bar{u}=u\circ\varphi_k^{-1}$, $\bar{\pi}=\pi\circ\varphi_k^{-1}$ and so on, the system \eqref{Stokes-surface-1} with respect to the local charts $(U_k,\varphi_k)$, $k\in\{1,\ldots,N\}$, reads as follows.
\begin{equation}
\label{Stokes-local-coordinates}
\begin{aligned}
\partial_t \bar{u}_{(k)}^\ell +\omega \bar{u}_{(k)}^\ell- \mu_s \bar{g}_{(k)}^{ij}\partial_i\partial_j\bar{u}_{(k)}^\ell + \bar{g}_{(k)}^{i\ell}\partial_i\bar{\pi}_{(k)}&=\bar{f}_{(k)}^\ell+F_{(k)}^\ell(\bar{u},\bar{\pi}) &&\text{in}\;\;\mathbb{R}^d \\
\partial_i \bar{u}_{(k)}^i &=H_{(k)}(\bar{u}) &&\text{in} \;\; \mathbb{R}^d \\
                       \bar{u}_{(k)}^\ell(0) &= 0  &&\text{in}\;\; \mathbb{R}^d,
\end{aligned}
\end{equation}
where
\begin{equation}\label{Def-Functions}
\bar{u}_{(k)}^\ell=(\bar{u}\bar{\psi}_k|\bar{\tau}_{(k)}^\ell),\ \bar{\pi}_{(k)}=\bar{\pi}\bar{\psi}_k,
\end{equation}
$$\bar{f}_{(k)}^\ell=(\bar{f}\bar{\psi}_k|\bar{\tau}_{(k)}^\ell),\quad F_{(k)}^\ell(\bar{u},\bar{\pi})=\bar{\pi} \bar{g}_{(k)}^{i\ell}\partial_i\bar{\psi}_k+(B_{(k)}\bar{u}|\bar{\tau}_{(k)}^\ell),$$
$\ell\in\{1,\ldots,d\}$, $B_{(k)}$ is a linear differential operator of order one and
$$H_{(k)}(\bar{u})=\bar{u}^i\partial_i\bar{\psi}_k-\bar{u}_{(k)}^j (\bar{\tau}_{(k)}^i|\partial_i\bar{\tau}_{(k)j}).$$
Here, upon translation and rotation, $\bar{g}_{(k)}^{ij}(0)=\delta_j^i$ and the coefficients have been extended in such a way that $|\bar{g}_{(k)}^{ij}-\delta_j^i|_{L_\infty(\mathbb{R}^d)}\le\eta$, where $\eta>0$ can be made as small as we wish, by decreasing the radius $r>0$ of the ball $B_{\mathbb{R}^d}(0,r)$.

In order to handle system \eqref{Stokes-local-coordinates}, we define vectors in $\R^d$ as follows:
$$\bar{u}_{(k)}:=(\bar{u}_{(k)}^1,\ldots,\bar{u}_{(k)}^d),\quad \bar{f}_{(k)}:=(\bar{f}_{(k)}^1,\ldots,\bar{f}_{(k)}^d)$$
and
$${F}_{(k)}(\bar{u},\bar{\pi}):=({F}_{(k)}^1(\bar{u},\bar{\pi}),\ldots,{F}_{(k)}^d(\bar{u},\bar{\pi})).$$
Moreover, we define the matrix $G_{(k)}=(\bar{g}_{(k)}^{ij})_{i,j=1}^d\in\R^{d\times d}$. With these notations, system \eqref{Stokes-local-coordinates} reads as
\begin{equation}
\label{Stokes-local-coordinates1}
\begin{aligned}
\partial_t \bar{u}_{(k)} +\omega \bar{u}_{(k)}- \mu_s (G_{(k)}\nabla|\nabla)\bar{u}_{(k)} + G_{(k)}\nabla\bar{\pi}_{(k)}&=\bar{f}_{(k)}+F_{(k)}(\bar{u},\bar{\pi}) &&\text{in}\;\;\mathbb{R}^d \\
{\rm div}\ \bar{u}_{(k)} &=H_{(k)}(\bar{u}) &&\text{in} \;\; \mathbb{R}^d \\
                       \bar{u}_{(k)}(0) &= 0  &&\text{in}\;\; \mathbb{R}^d.
\end{aligned}
\end{equation}
For each $k\in\{1,\ldots,N\}$, we define operators $L_{k,\omega}$ by the first two lines on the left side of \eqref{Stokes-local-coordinates1}. Then each operator is invertible and bounded. This can be seen by first freezing the coefficients at $x=0$, leading to full-space Stokes problems, which enjoy the property of $L_p$-$L_q$-maximal regularity by \cite[Theorem 7.1.1]{PrSi16}. Secondly, a Neumann series argument yields the claim, since $G_{(k)}$ is a perturbation of the identity in $\R^{d\times d}$. With the operator $L_{k,\omega}$ at hand, we may rewrite \eqref{Stokes-local-coordinates1} in the more condensed form
$$L_{k,\omega}(\bar{u}_{(k)},\bar{\pi}_{(k)})=(F_{(k)},H_{(k)})+(\bar{f}_{(k)},0),$$
Next, we remove the term $H_{(k)}$, since it is not of lower order. For that purpose, solve the equation ${\rm div}( G_{(k)}\nabla \phi_k)=H_{(k)}(\bar{u})$. Since $\int_{\R^d}H_{(k)}(\bar{u})dx=0$ ($H_{(k)}(\bar{u})$ is compactly supported), there exists a solution $\phi_k$ such that $\nabla \phi_k$ is unique, with regularity
$$\nabla\phi_k\in{_0}H_{p,\mu}^1(\R_+;H_q^1(\R^d)^d)\cap L_{p,\mu}(\R_+;H_q^3(\R^d)^d).$$
Moreover, we have the estimates
\begin{align}
\begin{split}\label{Est-phi}
|\nabla\phi_k|_{L_{p,\mu}(H_q^1(\R^d))}&\le C|\bar{u}|_{\mathbb{E}_{0,\mu}(\R^d)}\\
|\nabla\phi_k|_{\mathbb{E}_{1,\mu}(\R^d)}+|\nabla^2\phi_k|_{\mathbb{E}_{1,\mu}(\R^d)}&\le C|\bar{u}|_{\mathbb{E}_{1,\mu}(\R^d)}\\
|\nabla\phi_k|_{L_{p,\mu}(H_q^2(\R^d))}&\le C\omega^{-1/2}(\omega|\bar{u}|_{\mathbb{E}_{0,\mu}(\R^d)}+|\bar{u}|_{\mathbb{E}_{1,\mu}(\R^d)}),
\end{split}
\end{align}
see also \cite[(7.41)]{PrSi16}. Define
$$\tilde{u}_{(k)}=\bar{u}_{(k)}-G_{(k)}\nabla\phi_k-{u}_{(k)}^0\quad\text{and}\quad\tilde{\pi}_{(k)}=\bar{\pi}_{(k)}+(\partial_t+\omega)\phi_k-\Phi_k-{\pi}_{(k)}^0,$$
where $({u}_{(k)}^0,{\pi}_{(k)}^0)=L_{k,\omega}^{-1}(\bar{f}_{(k)},0)$ and $\Phi_k$ satisfies ${\rm div}( G_{(k)}\nabla \Phi_k)={\rm div}\tilde{F}_{(k)}(\bar{u},\bar{\pi})$ in $\dot{H}_{q}^{-1}(\R^d)$, with
$$\tilde{F}_{(k)}(\bar{u},\bar{\pi}):=F_{(k)}(\bar{u},\bar{\pi})+\mu_s (G_{(k)}\nabla|\nabla)(G_{(k)}\nabla\phi_k).$$
The couple $(\tilde{u}_{(k)},\tilde{\pi}_{(k)})$ then solves the equation
$$L_{k,\omega}(\tilde{u}_{(k)},\tilde{\pi}_{(k)})=(\tilde{F}_{(k)}(\bar{u},\bar{\pi})-G_{(k)}\nabla\Phi_k,0).$$
We note on the go that ${\rm div}(\tilde{F}_{(k)}(\bar{u},\bar{\pi})-G_{(k)}\nabla\Phi_k)=0$ and that the pressure $\tilde{\pi}_{(k)}$ enjoys additional time regularity. This can be seen exactly as in the proof of \cite[Proposition 7.3.5 (ii)]{PrSi16} with an obvious modification concerning the matrix $G_{(k)}$. In particular, there exists a constant $C>0$, such that
\begin{equation}\label{equation:ptilde}
|P_0\tilde{\pi}_{(k)}|_{L_{p,\mu}(L_q(B_{\R^d}(0,r)))}\le C|\tilde{u}_{(k)}|_{L_{p,\mu}(H_q^1(\R^d))}.
\end{equation}
Let us now introduce a norm for the solution, taking the parameter $\omega$ into account. Set
$$\|(\bar{u}_{(k)},\bar{\pi}_{(k)})\|_\omega=\omega |\bar{u}_{(k)}|_{\mathbb{E}_{0,\mu}(\R^d)}+|\bar{u}_{(k)}|_{\mathbb{E}_{1,\mu}(\R^d)}+|\nabla\bar{\pi}_{(k)}|_{\mathbb{E}_{0,\mu}(\R^d)},$$
and similarly for $\|(u,\pi)\|_\omega$ on $\Sigma$.

For each $k\in\{1,\ldots,N\}$ there exists $\omega_0>0$ such that the operator $L_{k,\omega}$ has the property of $L_p$-$L_q$-maximal regularity, provided $\omega>\omega_0$. In particular, there exists a constant $C>0$ such that
$$\|({u}_{(k)}^0,{\pi}_{(k)}^0)\|_\omega\le C|\bar{f}_{(k)}|_{\mathbb{E}_{0,\mu}(\R^d)}\le C|f|_{\mathbb{E}_{0,\mu}(\Sigma)}$$
and
$$
\|(\tilde{u}_{(k)},\tilde{\pi}_{(k)})\|_\omega\le C\omega^{-\gamma}\|(u,\pi)\|_\omega,
$$
where for the last inequality, we made use of Proposition \ref{Proposition-Reg-Pressure}, implying the estimate
$$\|\bar{\pi}G_{(k)}\nabla\bar{\psi}_k\|_{\mathbb{E}_{0,\mu}(\R^d)}\le C\omega^{-\gamma}(\omega|u|_{\mathbb{E}_{0,\mu}(\Sigma)}+|u|_{\mathbb{E}_{1,\mu}(\Sigma)}),$$
for some constant $\gamma>0$, by interpolation between 
$$L_{p,\mu}(\R_+;L_q(\Sigma))\quad\text{and}\quad  L_{p,\mu}(\R_+;H_q^2(\Sigma)).$$
In the same way, making also use of \eqref{Est-phi} and the definition of $\Phi_k$, we obtain
$$|\nabla\Phi_k|_{\mathbb{E}_{0,\mu}(\R^d)}\le C\omega^{-\gamma}\|(u,\pi)\|_\omega.$$
Furthermore, we have
\begin{align*}
|(\partial_t+\omega)\phi_k|_{\mathbb{E}_{0,\mu}(B_{\mathbb{R}^d}(0,r))}&\le |\tilde{\pi}_{(k)}|_{\mathbb{E}_{0,\mu}(B_{\mathbb{R}^d}(0,r))}+|\bar{\pi}_{(k)}|_{\mathbb{E}_{0,\mu}(B_{\mathbb{R}^d}(0,r))}\\
&\quad+|\pi_{(k)}^0+\Phi_k|_{\mathbb{E}_{0,\mu}(B_{\mathbb{R}^d}(0,r))}\\
&\le |\tilde{\pi}_{(k)}|_{\mathbb{E}_{0,\mu}(B_{\mathbb{R}^d}(0,r))}+|\bar{\pi}_{(k)}|_{\mathbb{E}_{0,\mu}(B_{\mathbb{R}^d}(0,r))}\\
&\quad+C(|\nabla\pi_{(k)}^0|_{\mathbb{E}_{0,\mu}(B_{\mathbb{R}^d}(0,r))}+|\nabla\Phi_{k}|_{\mathbb{E}_{0,\mu}(B_{\mathbb{R}^d}(0,r))})\\
&\le C(|f|_{\mathbb{E}_{0,\mu}(\Sigma)}+\omega^{-\gamma}\|(u,\pi)\|_\omega),
\end{align*}
by \eqref{equation:ptilde}, Proposition \ref{Proposition-Reg-Pressure} for $\bar{\pi}$ and the Poincar\'{e} inequality for $\pi_{(k)}^0$ and $\Phi_k$, since we may assume without loss of generality, that $\pi_{(k)}^0$ as well as $\Phi_k$ have mean value zero on $B_{\mathbb{R}^d}(0,r)$. By interpolation with \eqref{Est-phi} this yields
$$|(\partial_t+\omega)\nabla\phi_k|_{\mathbb{E}_{0,\mu}(B_{\mathbb{R}^d}(0,r))}\le C(|f|_{\mathbb{E}_{0,\mu}(\Sigma)}+\omega^{-\gamma/2}\|(u,\pi)\|_\omega).$$
In conclusion, we obtain the estimate
$$\|\bar{\chi}_k(\bar{u}_{(k)},\bar{\pi}_{(k)})\|_\omega\le C(|f|_{\mathbb{E}_{0,\mu}(\Sigma)}+\omega^{-\gamma}\|(u,\pi)\|_\omega)$$
valid for each $k\in\{1,\ldots,N\}$ and $C>0$ does not depend on $\omega>0$. Here, $\{\chi_k\}_{k=1}^N\subset C^\infty(\Sigma)$ such that $\chi_k=1$ on ${\rm supp}(\psi_k)$ and ${\rm supp}(\chi_k)\subset U_k$. As usual, we have set $\bar{\chi}_k=\chi_k\circ\varphi^{-1}_k$.

For the components $\bar{u}_{(k)}^\ell\circ\varphi_k$ of the vector $\bar{u}_{(k)}\circ\varphi_k\in\R^d$ we derive from \eqref{Def-Functions}
$$\bar{u}_{(k)}^\ell\circ\varphi_k=(u\psi_k|\tau_{(k)}^\ell)=(u^j\tau_{(k)j}\psi_k|\tau_{(k)}^\ell)=u^\ell\psi_k,$$
hence
$$u=\sum_{k=1}^N\psi_k u=\sum_{k=1}^N\psi_k u^\ell\tau_{(k)\ell}=\sum_{k=1}^N(\bar{u}_{(k)}^\ell\bar{\tau}_{(k)\ell})\circ\varphi_k.$$
Since $\chi_k=1$ on ${\rm supp}(\psi_k)$, this finally yields the estimate
$$\|(u,\pi)\|_\omega\le C(|f|_{\mathbb{E}_{0,\mu}(\Sigma)}+\omega^{-\gamma}\|(u,\pi)\|_\omega)$$
valid for all $\omega>\omega_0$. Choosing $\omega_0>0$ sufficiently large, we conclude
$$\|(u,\pi)\|_\omega\le C|f|_{\mathbb{E}_{0,\mu}(\Sigma)}.$$
This in turn implies uniqueness of a solution to \eqref{Stokes-surface-1}.

It remains to prove the existence of a solution to \eqref{Stokes-surface-1}. To this end, we may assume that $({\rm div}_\Sigma f,g,u_0)=0$. Solve the parabolic problem
\begin{equation}
\label{Parabolic-surface-1}
\begin{aligned}
\partial_t v+\omega v - \mu_s \Delta_\Sigma v&=f &&\text{on}\;\;\Sigma \\
                      v(0) &= 0  &&\text{on}\;\; \Sigma,
\end{aligned}
\end{equation}
by \cite[Theorem 6.4.3 (ii)]{PrSi16} to obtain a unique solution $v\in {_0}\mathbb{E}_{1,\mu}(\Sigma)$. Next, we solve $\Delta_{\Sigma}^L\phi={\rm div}_{\Sigma} v$ by \cite[Theorem 6.4.3 (i)]{PrSi16} to obtain a solution $\phi$ such that $\nabla_\Sigma\phi$ is unique with regularity
$$\nabla_\Sigma\phi\in{_0}H_{p,\mu}^1(\R_+;L_q(\Sigma,\mathsf{T}\Sigma))\cap L_{p,\mu}(\R_+;H_q^2(\Sigma,\mathsf{T}\Sigma)).$$
Define $\tilde{u}=v-\nabla_\Sigma\phi$ and $\tilde{\pi}=(\partial_t+\omega)\phi$. It follows that
$$L_\omega(\tilde{u},\tilde{\pi})=f+\mu_s\Delta_\Sigma\nabla_\Sigma\phi,$$
where
$$L_\omega:\{u\in{_0}\mathbb{E}_{1,\mu}(\Sigma):{\rm div}_\Sigma u=0\}\times L_{p,\mu}(\R_+;\dot{H}_q^1(\Sigma))\to \mathbb{E}_{0,\mu}(\Sigma)$$
is defined by
$$L_\omega(u,\pi)=
  \partial_t u+\omega u - \mu_s \Delta_\Sigma u + \nabla_\Sigma \pi.$$
Making use of local coordinates, one can show that
$$\Delta_\Sigma\nabla_\Sigma\phi=\nabla_\Sigma\Delta_{\Sigma}^L\phi+\mathcal{A}_\Sigma\phi,$$
where $\mathcal{A}_\Sigma$ is a second order operator. Setting $\hat{u}=\tilde{u}$, $\hat{\pi}=\tilde{\pi}-\mu_s\Delta_{\Sigma}^L\phi$ and $Sf:=(\hat{u},\hat{\pi})$, we obtain
$$L_\omega Sf=L_\omega(\hat{u},\hat{\pi})=f+Rf,$$
with $Rf:=\mu_s\mathcal{A}_\Sigma\phi$. Since $\mathcal{A}_\Sigma$ is of second order, this yields
\begin{multline*}
|Rf|_{\mathbb{E}_{0,\mu}(\Sigma)}=|\mu_s\mathcal{A}_\Sigma\phi|_{\mathbb{E}_{0,\mu}(\Sigma)}\le C|\phi|_{L_{p,\mu}(\R_+;H_q^2(\Sigma))}=\\=|(\Delta_\Sigma^L)^{-1}\Delta_\Sigma^L\phi|_{L_{p,\mu}(\R_+;H_q^2(\Sigma))}\le
C|{\rm div}_\Sigma v|_{\mathbb{E}_{0,\mu}(\Sigma)}\le \omega^{-1/2}C|f|_{\mathbb{E}_{0,\mu}(\Sigma)},
\end{multline*}
where the constant $C>0$ does not depend on $\omega$. We note that $(\Delta_\Sigma^L)^{-1}$, acting on functions with average zero, is well defined.
A Neumann series argument implies that $(I+R)$ is invertible provided $\omega>0$ is sufficiently large. Hence the operator $S(I+R)^{-1}$ is a right inverse for $L_\omega$, which means that $L_\omega$ is surjective. This completes the proof of Theorem \ref{Lin-Theorem-1}.

\subsection{The surface Stokes operator}\label{section:Surface-Stokes}

By Proposition \ref{pro:D-Sigma},  $\mathcal{P}_\Sigma{\rm div}_\Sigma\mathcal{D}_\Sigma(u)$ is a lower perturbation of $\Delta_\Sigma u$ if ${\rm div}_\Sigma u$ is prescribed. This implies the following result for the system
\begin{equation}
\label{Stokes-surface-full}
\begin{aligned}
\partial_t u+\omega u - 2\mu_s \mathcal{P}_\Sigma{\rm div}_\Sigma\mathcal{D}_\Sigma(u) + \nabla_\Sigma \pi&=f &&\text{on}\;\;\Sigma \\
{\rm div}_\Sigma u &=g &&\text{on} \;\; \Sigma \\
                       u(0) &= u_0  &&\text{on}\;\; \Sigma.
\end{aligned}
\end{equation}
\begin{corollary}\label{Cor-Stokes}
Under the assumptions in Theorem \ref{Lin-Theorem-1}, there exists $\omega_0>0$ such that for each $\omega>\omega_0$, problem \eqref{Stokes-surface-full} admits a unique solution
$$u\in H_{p,\mu}^1(\R_+;L_q(\Sigma,\mathsf{T}\Sigma))\cap L_{p,\mu}(\R_+;H_q^2(\Sigma,\mathsf{T}\Sigma)),\quad \pi\in L_{p,\mu}(\R_+;\dot{H}_q^1(\Sigma)),$$
if and only if the data $(f,g,u_0)$ are subject to the conditions (1)-(4) in Theorem \ref{Lin-Theorem-1}. Moreover, the solution $(u,\pi)$ depends continuously on the given data $(f,g,u_0)$ in the corresponding spaces.
\end{corollary}
\begin{proof}
Without loss of generality, we may assume $({\rm div}_\Sigma f,g,u_0)=0$. With the operator $L_\omega$ defined above, we rewrite $\eqref{Stokes-surface-full}$ as
$$(u,\pi)=L_\omega^{-1}f+\mu_sL_\omega^{-1}(\kappa_\Sigma L_\Sigma-L_\Sigma^2)u.$$
For the term of order zero on the right hand side, we have the estimate
$$|(\kappa_\Sigma L_\Sigma-L_\Sigma^2)u|_{\mathbb{E}_{0,\mu}(\Sigma)}\le \omega^{-1}C\|(u,\pi)\|_\omega,$$
where $C>0$ does not depend on $\omega>0$. By Theorem \ref{Lin-Theorem-1} the solution depends continuously on the data, hence there exists a constant $M=M(\omega_0)>0$ such that
$$\|L_\omega^{-1}(\kappa_\Sigma L_\Sigma-L_\Sigma^2)u\|_\omega\le M|(\kappa_\Sigma L_\Sigma-L_\Sigma^2)u|_{\mathbb{E}_{0,\mu}(\Sigma)}\le \omega^{-1}MC\|(u,\pi)\|_\omega.$$
Therefore, a Neumann series argument yields the claim, if $\omega>0$ is chosen sufficiently large.
\end{proof}

We will now define the Stokes operator on surfaces. Let $P_{H,\Sigma}$ denote the \emph{surface Helmholtz projection}, defined by
$$P_{H,\Sigma} v:=v-\nabla_\Sigma\psi,\quad v\in L_q(\Sigma,\mathsf{T}\Sigma),$$
where $\nabla_\Sigma\psi\in L_q(\Sigma,\mathsf{T}\Sigma)$ is the unique solution of
$$(\nabla_\Sigma\psi|\nabla_\Sigma\phi)_\Sigma=(v|\nabla_\Sigma\phi)_\Sigma,\quad \phi\in\dot{H}_{q'}^1(\Sigma).$$
We note that $(P_{H,\Sigma}u|v)_\Sigma=(u|P_{H,\Sigma} v)_\Sigma$ for all $u\in L_q(\Sigma,\mathsf{T}\Sigma)$, $v\in L_{q'}(\Sigma,\mathsf{T}\Sigma)$, which follows directly from the definition of $P_{H,\Sigma}$ (and for smooth functions from the
surface divergence theorem \eqref{divergence-theorem}).
Define 
$$X_0:=L_{q,\sigma}(\Sigma,\mathsf{T}\Sigma):=P_{H,\Sigma}L_q(\Sigma,\mathsf{T}\Sigma)$$
and $X_1:=H_q^2(\Sigma,\mathsf{T}\Sigma)\cap L_{q,\sigma}(\Sigma,\mathsf{T}\Sigma)$.
The \emph{surface Stokes operator} is defined by
\begin{equation}\label{eq:StokesOp}
A_{S,\Sigma}u:=-2 \mu_sP_{H,\Sigma} \mathcal{P}_\Sigma{\rm div}_\Sigma\mathcal{D}_\Sigma(u),\quad u\in D(A_{S,\Sigma}):=X_1.
\end{equation}
We would also like to refer to the survey article~\cite{HiSa18} for the Stokes operator in various other geometric settings.

Making use of the projection $P_{H,\Sigma}$, \eqref{Stokes-surface-full} with $({\rm div}_\Sigma f,g)=0$ is equivalent to the equation
\begin{equation}\label{Abstract-Stokes}
\partial_t u+\omega u+A_{S,\Sigma}u=f,\quad t>0,\quad u(0)=u_0.
\end{equation}
By Corollary \ref{Cor-Stokes}, the operator $A_{S,\Sigma}$ has $L_{p}$-maximal regularity, hence $-A_{S,\Sigma}$ generates an analytic $C_0$-semigroup in $X_0$, see for instance \cite[Proposition 3.5.2]{PrSi16}.

\subsection{Interpolation spaces}\mbox{}
In this subsection, we will determine
the real and complex interpolation spaces $(X_0,X_1)_{\alpha,p}$ and $(X_0,X_1)_\alpha$, respectively. To this end, let $A_\Sigma u:=-2\mu_s\cP_\Sigma{\rm div}_\Sigma\cD_\Sigma(u)$ with domain $D(A_\Sigma):=H_q^2(\Sigma,\mathsf{T}\Sigma)$ and
define a linear mapping $Q$ on $D(A_\Sigma)$ by
$$Q=(\omega+A_{S,\Sigma})^{-1}P_{H,\Sigma} (\omega+A_\Sigma),$$
for some fixed and sufficiently large $\omega>0$.

Then $Q:D(A_\Sigma)\to X_1$ is a bounded projection,  as $Qu\in X_1$ and 
$$Q^2u=(\omega+A_{S,\Sigma})^{-1}P_{H,\Sigma} (\omega+A_\Sigma)Qu=(\omega+A_{S,\Sigma})^{-1}(\omega+A_{S,\Sigma})Qu=Qu,$$
for all $u\in D(A_\Sigma)$. Furthermore, $Q|_{X_1}=I_{X_1}$ and therefore $Q:D(A_\Sigma)\to X_1$ is surjective. By a duality argument, there exists some constant $C>0$ such that
\begin{equation}\label{eq:Q}
\|Qu\|_{L_q(\Sigma)}\le C\|u\|_{L_q(\Sigma)}
\end{equation}
for all $u\in D(A_\Sigma)$. In fact,
\begin{align*}
(Qu|\phi)_{\Sigma}&=((\omega+A_{S,\Sigma})^{-1}P_{H,\Sigma} (\omega+A_\Sigma)u|\phi)_\Sigma\\
&=(P_{H,\Sigma}(\omega+A_{S,\Sigma})^{-1}P_{H,\Sigma} (\omega+A_\Sigma)u|\phi)_\Sigma\\
&=(P_{H,\Sigma}(\omega+A_\Sigma)u| (\omega+A_{S,\Sigma})^{-1}P_{H,\Sigma}\phi)_{\Sigma}\\
&=(u|(\omega+A_\Sigma) (\omega+A_{S,\Sigma})^{-1}P_{H,\Sigma}\phi)_{\Sigma}
\end{align*}
implies
$$|(Qu|\phi)_{\Sigma}|\le C\|u\|_{L_q(\Sigma)}\|\phi\|_{L_{q'}(\Sigma)}$$
for all $u\in D(A_\Sigma)$ and $\phi\in L_{q'}(\Sigma,\mathsf{T}\Sigma)$, with 
$$C:=\|(\omega+A_\Sigma) (\omega+A_{S,\Sigma})^{-1}P_{H,\Sigma}\|_{\mathcal{B}(L_{q'}(\Sigma);L_{q'}(\Sigma))}>0.$$
Since $D(A_\Sigma)$ is dense in $L_q(\Sigma,\mathsf{T}\Sigma)$, there exists a unique bounded extension $\tilde{Q}:L_{q}(\Sigma,\mathsf{T}\Sigma)\to X_0$ of $Q$. Clearly, $\tilde{Q}$ is a projection and as $X_1$ is dense in $X_0$, $\tilde{Q}|_{X_0}=I_{X_0}$.

It follows that
$$L_q(\Sigma,\mathsf{T}\Sigma)=X_0\oplus N(\tilde{Q})\quad\text{and}\quad D(A_\Sigma)=X_1\oplus [D(A_\Sigma)\cap N(\tilde{Q})]$$
since $\tilde{Q}D(A_\Sigma)=D(A_\Sigma)\cap R(\tilde{Q})=D(A_{S,\Sigma})=X_1$.
Moreover, with the help of the projection $\tilde{Q}$ and the relation $R(\tilde{Q})=L_{q,\sigma}(\Sigma,\mathsf{T}\Sigma)$, we may now compute
\begin{multline*}
(X_0,X_1)_\alpha=(\tilde{Q} L_q(\Sigma,\mathsf{T}\Sigma),\tilde{Q}D(A_\Sigma))_\alpha\\
=\tilde{Q}(L_q(\Sigma,\mathsf{T}\Sigma),D(A_\Sigma))_\alpha=H_q^{2\alpha}(\Sigma,\mathsf{T}\Sigma)\cap L_{q,\sigma}(\Sigma,\mathsf{T}\Sigma)
\end{multline*}
as well as
\begin{multline*}
(X_0,X_1)_{\alpha,p}=(\tilde{Q} L_q(\Sigma,\mathsf{T}\Sigma),\tilde{Q}D(A_\Sigma))_{\alpha,p}\\
=\tilde{Q}(L_q(\Sigma,\mathsf{T}\Sigma),D(A_\Sigma))_{\alpha,p}=B_{qp}^{2\alpha}(\Sigma,\mathsf{T}\Sigma)\cap L_{q,\sigma}(\Sigma,\mathsf{T}\Sigma)
\end{multline*}
for $\alpha\in (0,1)$ and $p\in (1,\infty)$, see \cite[Theorem 1.17.1.1]{Tri78}.

\subsection{Nonlinear well-posedness}

We will show that there exists a unique local-in-time solution to \eqref{NS-surface-2}. Observe that the semilinear problem \eqref{NS-surface-2} is equivalent to the abstract semilinear evolution equation
\begin{equation}\label{Abstract-Navier-Stokes}
\partial_t u+A_{S,\Sigma}u=F_\Sigma(u),\quad t>0,\quad u(0)=u_0,
\end{equation}
where $F_\Sigma(u):=-P_{H,\Sigma}\mathcal{P}_\Sigma(u\cdot\nabla_\Sigma u)$. In order to solve this equation in the maximal regularity class $\mathbb{E}_{1,\mu}(\Sigma)$, we will apply Theorem 2.1 in \cite{LPW14}. To this end, let $q\in (1,d)$ and 
\begin{equation}\label{mu-c}
\mu_c:=\frac{1}{2}\left(\frac{d}{q}-1\right)+\frac{1}{p},
\end{equation}
with $2/p+d/q<3$, so that $\mu_c\in (1/p,1)$. We will show that for each $\mu\in (\mu_c,1]$
there exists $\beta\in (\mu-1/p,1)$ with $2\beta-1<\mu-1/p$ such that $F_\Sigma$ satisfies the estimate
\begin{equation}\label{ineq-F-Sigma}
|F_\Sigma(u)-F_\Sigma(v)|_{X_0}\le C(|u|_{X_\beta}+|v|_{X_\beta})||u-v|_{X_\beta}
\end{equation}
for all $u,v\in X_\beta:=(X_0,X_1)_{\beta,p}$. 

By H\"olders inequality, the estimate
$$|F_\Sigma(u)|_{L_q(\Sigma,\mathsf{T}\Sigma)}\le C|u|_{L_{qr'}(\Sigma,\mathsf{T}\Sigma)}|u|_{H_{qr}^1(\Sigma,\mathsf{T}\Sigma)}$$
holds. We choose $r,r'\in (1,\infty)$ in such a way that 
$$1-\frac{d}{qr}=-\frac{d}{qr'}\quad\text{or equivalently}\quad \frac{d}{qr}=\frac{1}{2}\left(1+\frac{d}{q}\right),$$
which is feasible if $q\in (1,d)$. Next, by Sobolev embedding, we have
$$(X_0,X_1)_{\beta,p}\subset B_{qp}^{2\beta}(\Sigma,\mathsf{T}\Sigma)\hookrightarrow H_{qr}^1(\Sigma,\mathsf{T}\Sigma)\cap L_{qr'}(\Sigma,\mathsf{T}\Sigma),$$
provided 
$$2\beta-\frac{d}{q}>1-\frac{d}{qr}\quad\text{or equivalently}\quad
\beta>\frac{1}{4}\left(\frac{d}{q}+1\right).$$
The condition $\beta<1$ requires $q>d/3$, hence $q\in (d/3,d)$. Note that
$$\frac{1}{2}\left(\frac{d}{q}+1\right)-1<\mu-1/p,$$
since $\mu>\mu_c$.
This implies that $1>\beta>(d/q+1)/4$ can be chosen in such a way that the inequalities $2\beta-1<\mu-1/p$ and $\mu-1/p<\beta$ are satisfied.

In case $q\ge d$, we may choose any $\mu\in (1/p,1]$, since 
$$B_{qp}^{2\beta}(\Sigma,\mathsf{T}\Sigma)\hookrightarrow H_{q}^1(\Sigma,\mathsf{T}\Sigma)\cap L_{\infty}(\Sigma,\mathsf{T}\Sigma),$$
provided $2\beta>1$.

Since $F_\Sigma$ is bilinear, it follows that the estimate \eqref{ineq-F-Sigma} holds and, moreover, that $F_\Sigma\in C^\infty(X_\beta,X_0)$. Therefore, Theorem 2.1 in \cite{LPW14} yields the following result.
\begin{theorem}\label{Theorem-LWP}
\mbox{} 
Let $p,q\in (1,\infty)$.  Suppose that one of the following conditions hold:
\begin{itemize}
\item[(a)]
$q\in (d/3,d)$, $2/p+d/q<3$ and $\mu\in (\mu_c,1],$ where $\mu_c$ is defined in \eqref{mu-c}.
\vspace{1mm}
\item[(b)] 
 $q\ge d$ and  $\mu\in (1/p,1]$.
\end{itemize}
Then, for any initial value $u_0\in B_{qp}^{2\mu-2/p}(\Sigma,\mathsf{T}\Sigma)\cap L_{q,\sigma}(\Sigma,\mathsf{T}\Sigma)$, 
there exists a number $a=a(u_0)>0$ such that \eqref{NS-surface-2} admits a unique solution
$$u\in H_{p,\mu}^1((0,a);L_q(\Sigma,\mathsf{T}\Sigma))\cap L_{p,\mu}((0,a);H_q^2(\Sigma,\mathsf{T}\Sigma)),\quad \pi\in L_{p,\mu}((0,a);\dot{H}_{q}^1(\Sigma)).$$
Moreover,
$$u\in C([0,a];B_{qp}^{2\mu-2/p}(\Sigma,\mathsf{T}\Sigma))\cap C((0,a];B_{qp}^{2-2/p}(\Sigma,\mathsf{T}\Sigma)).$$
\end{theorem}

\begin{remark}\mbox{}
\begin{itemize}
\item[(a)] The number $\mu_c\in (1/p,1]$ defined in \eqref{mu-c} is called the critical weight and was introduced in
\cite{PSW18, PrWi17,PrWi18}.
It has been shown in~\cite{PSW18} that the `critical spaces' $(X_0,X_1)_{\mu_c-1/p,p}$ correspond 
to scaling invariant spaces in case the underlying equations enjoy scaling invariance.
\item[(b)] In future work, we plan to show that $A_{S,\Sigma}$ has a bounded $H^\infty$-calculus. Then one may set $\mu=\mu_c$ in Theorem~\ref{Theorem-LWP}, thereby obtaining well-posedness in critical spaces. 
\item[(c)] In case $d=2$, global existence has been obtained by Taylor \cite[Proposition 6.5]{Tay92}. An alternative proof can be based on the approach via critical spaces mentioned above.

\end{itemize}
\end{remark}

\section{Stability of equilibria, examples}

Consider the semilinear evolution equation
\begin{equation}\label{Abstract-Navier-Stokes2}
\partial_t u+A_{S,\Sigma}u=F_\Sigma(u),\ t>0,\quad u(0)=u_0,
\end{equation}
in $X_0=L_{q,\sigma}(\Sigma,\mathsf{T}\Sigma)$.
Define the set
\begin{equation}\label{u-equilibria}
\cE:=\{u_*\in H_q^2(\Sigma;\mathsf{T}\Sigma):{\rm div}_\Sigma u_*=0,\ \mathcal{D}_\Sigma(u_*)=0\}.
\end{equation}
We show that the set $\cE$ corresponds exactly to the set of equilibria for \eqref{Abstract-Navier-Stokes2}. To this end, let $u_*$ be an equilibrium of \eqref{Abstract-Navier-Stokes2}, i.e.\ $u_*\in X_1$ satisfies $A_{S,\Sigma}u_*=F_\Sigma(u_*)$. Multiplying this equation by $u_*$ and integrating over $\Sigma$ yields
$$0=2\mu_s({\rm div}_\Sigma\mathcal{D}_\Sigma(u_*)|u_*)_\Sigma+(u_*\cdot\nabla_\Sigma u_*|u_*)_\Sigma.$$
By Lemma \ref{lem:app} and the surface divergence theorem \eqref{divergence-theorem}, the last term vanishes, since ${\rm div}_\Sigma u_*=0$.  Furthermore, \eqref{divergence-parts} and again \eqref{divergence-theorem} show that
$$({\rm div}_\Sigma\mathcal{D}_\Sigma(u_*)|u_*)_\Sigma=-\int_\Sigma|\mathcal{D}_\Sigma(u_*)|^2 \,d\Sigma,$$
which implies $\mathcal{D}_\Sigma(u_*)=0$, hence $u_*\in\cE$.

Conversely, let $u_*\in\cE$ be given. Then $A_{S,\Sigma}u_*=0$ and from Lemma \ref{lem:app} we obtain that
$$F_\Sigma(u_*)=-P_{H,\Sigma}\cP_\Sigma(u_*\cdot\nabla_\Sigma u_*)=\frac{1}{2}P_{H,\Sigma}(\nabla_\Sigma|u_*|^2)=0.$$
Summarizing, we have shown that
$$\cE=\{u_*\in H_q^2(\Sigma;\mathsf{T}\Sigma):{\rm div}_\Sigma u_*=0,\ A_{S,\Sigma}u_*=F_\Sigma(u_*)\}.$$
Observe that the set $\cE$ is a linear manifold, consisting exactly of the \emph{Killing fields} on $\Sigma$,  see Remark \ref{rem:app}.

Define an operator $A_0:X_1\to X_0$ by
\begin{equation}\label{eq-full-lin}
A_0v=A_{S,\Sigma}v-F_\Sigma'(u_*)v,
\end{equation}
where $F_\Sigma'(u_*)v:=-P_{H,\Sigma}\mathcal{P}_\Sigma\left(v\cdot\nabla_\Sigma u_*+u_*\cdot\nabla_\Sigma v\right)$.  This operator is the full linearization of \eqref{Abstract-Navier-Stokes2} at the equilibrium $u_*\in\cE$.
We collect some properties of $A_0$ in the following
\begin{proposition}\label{pro:normally-stable}
Suppose $u_*\in\cE$ and let $A_0$ be given by  \eqref{eq-full-lin}. Then $-A_0$ generates a compact analytic $C_0$-semigroup in $X_0$ which has $L_p$-maximal regularity. The spectrum of $A_0$ consists only of eigenvalues of finite algebraic multiplicity and the kernel $N(A_0)$ is given by
$$N(A_0)=\{u\in H_q^2(\Sigma,\mathsf{T}\Sigma):{\rm div}_\Sigma u=0,\ \mathcal{D}_\Sigma(u)=0\}.$$
If $\cE\neq\{0\}$, then $u_*\in\cE$ is normally stable, i.e.
\begin{enumerate}
\item[(i)] ${\rm Re}\ \sigma(-A_0)\le 0$ and $\sigma(A_0)\cap i\R=\{0\}$.
\item[(ii)] $\lambda=0$ is a semi-simple eigenvalue of $A_0$.
\item[(iii)] The kernel $N(A_0)$ is isomorphic to $\mathsf{T}_{u_*}\cE$.
\end{enumerate}
In case $\cE=\{0\}$, it holds that ${\rm Re}\ \sigma(-A_0)<0$.
\end{proposition}
\begin{proof}
By Subsection \ref{section:Surface-Stokes} the surface Stokes operator $A_{S,\Sigma}$ has the property of $L_p$-maximal regularity in $X_0$, hence $-A_{S,\Sigma}$ is $\mathcal{R}$-sectorial in $X_0$. Furthermore, the linear mapping $[v\mapsto F_\Sigma'(u_*)v]$ is relatively bounded with respect to $A_{S,\Sigma}$. An application of \cite[Proposition 4.4.3]{PrSi16} yields that $-A_0$ generates an analytic $C_0$-semigroup in $X_0$ having $L_p$-maximal regularity. Since the domain $X_1$ of $A_0$ is compactly embedded into $X_0$, the spectrum $\sigma(A_0)$ is discrete and consists solely of eigenvalues of $A_0$ having finite algebraic multiplicity. 

We first consider the case $\mathcal{E}\neq\{0\}$. Let $\lambda\in\sigma(-A_0)$ and denote by $v\in X_1$ a corresponding eigenfunction. Multiplying the equation $\lambda v+A_0v=0$ by the complex conjugate $\bar{v}$ and integrating over $\Sigma$ yields
\begin{align}\label{eq:spectralest}
{\rm Re}\lambda|v|_{L_2(\Sigma)}^2&=2\mu_s{\rm Re}(\mathcal{P}_\Sigma{\rm div}_\Sigma\mathcal{D}_\Sigma(v)|\bar{v})_\Sigma-{\rm Re}(\mathcal{P}_\Sigma(v\cdot\nabla_\Sigma u_*)+\mathcal{P}_\Sigma(u_*\cdot\nabla_\Sigma v)|\bar{v})_\Sigma\nn\\
&=-2\mu_s\int_\Sigma |\mathcal{D}_\Sigma(v)|^2\,d\Sigma.
\end{align}
Here, we used the identities
$$(v\cdot\nabla_\Sigma u_*|\bar{v})_\Sigma=-\overline{(v\cdot\nabla_\Sigma u_*|\bar{v})_\Sigma}$$
and
\begin{align*}
(u_*\cdot\nabla_\Sigma v|\bar{v})_\Sigma&=(u_*|\nabla_\Sigma|v|^2)_\Sigma-
(u_*|(\nabla_\Sigma\bar{v})v)_\Sigma\\
&=-\overline{(u_*\cdot\nabla_\Sigma v|\bar{v})_\Sigma}
\end{align*}
since $\mathcal{D}_\Sigma(u_*)=0$ and ${\rm div}_\Sigma u_*=0$, employing the surface divergence theorem \eqref{divergence-theorem}. It follows that ${\rm Re}\lambda\le 0$ and if ${\rm Re}\lambda=0$, then $\mathcal{D}_\Sigma(v)=0$. Observe that the equations $\mathcal{D}_\Sigma(v)=0=\mathcal{D}_\Sigma(u_*)$ then lead to the identity
$$\mathcal{P}_\Sigma(v\cdot\nabla_\Sigma u_*+u_*\cdot\nabla_\Sigma v)=-\nabla_\Sigma(u_*|v),$$
hence $F_\Sigma'(u_*)v=P_{H,\Sigma}(\nabla_\Sigma(u_*|v))=0$ and therefore $A_0v=0$.


The above calculations show that $\sigma(A_0)\cap i\R=\{0\}$ and
$$N(A_0)=\{v\in X_1:\mathcal{D}_\Sigma(v)=0\},$$
wherefore $N(A_0)\cong\cE\cong \mathsf{T}_{u_*}\cE$.

We will now prove that $\lambda=0$ is semi-simple. To this end, it suffices to prove that $N(A_0^2)\subset N(A_0)$. Let $w\in N(A_0^2)$ and $v:=A_0w$. Then $v\in N(A_0)$ and we obtain
$$|v|_{L_2(\Sigma)}^2=(A_0w|v)_\Sigma=2\mu_s\int_\Sigma \mathcal{D}_\Sigma(w):\mathcal{D}_\Sigma(v)\,d\Sigma=0,$$
by Lemma \ref{lem:app}, \eqref{divergence-parts} and the property $\mathcal{D}_\Sigma(v)=0$, which implies $w\in N(A_0)$.

Finally, we consider the case $\mathcal{E}=\{0\}$. If $\lambda\in \sigma(-A_0)$ with eigenfunction $v\neq 0$, it follows from \eqref{eq:spectralest} that $\operatorname{Re}\lambda\le 0$ and if $\operatorname{Re}\lambda=0$, then $\mathcal{D}_\Sigma(v)=0$ by \eqref{eq:spectralest}, hence $v\in\mathcal{E}=\{0\}$, a contradiction. Therefore, in this case, ${\rm Re}\ \sigma(-A_0)<0$.
\end{proof}
\begin{remark}\label{schief}
The above computations show that the operator $A_0$ from \eqref{eq-full-lin} is not necessarily symmetric in case $u_*\neq 0$. In fact, we have
\begin{equation*}
(A_0 v|\bar v)_\Sigma = 2\mu_s\int_\Sigma |\cD_\Sigma(v)|^2\,d\Sigma 
+i\,{\rm Im}[(v\cdot \nabla_\Sigma u_*|\bar v)_\Sigma + (u_*\cdot \nabla_\Sigma v | \bar v)_\Sigma],
\; v\in X_1.
\end{equation*}
\end{remark}
Since $F_\Sigma$ is bilinear, we obtain from \cite[Theorem 5.3.1]{PrSi16} and the proof of Proposition 5.1 as well as Theorem 5.2 in \cite{MPS19} the following result. 
\begin{theorem}\label{thm:stability}
Suppose $p,q$ and $\mu$ satisfy the assumptions of Theorem~\ref{Theorem-LWP}.

Then each equilibrium $u_*\in\cE$ is stable in $X_{\gamma,\mu}:=B_{qp}^{2\mu-2/p}(\Sigma,\mathsf{T}\Sigma)\cap L_{q,\sigma}(\Sigma,\mathsf{T}\Sigma)$ and there exists $\delta>0$ such that the unique solution $u(t)$ of \eqref{Abstract-Navier-Stokes2} with initial value $u_0\in X_{\gamma,\mu}$ satisfying $|u_0-u_*|_{X_{\gamma,\mu}}<\delta$ exists on $\R_+$ and converges at an exponential rate in $X_{\gamma,1}$ to a (possibly different) equilibrium $u_\infty\in\cE$ as $t\to\infty$.
\end{theorem}


\subsection{Existence of equilibria}\label{Killing}
According to  \eqref{u-equilibria}, see also Proposition~\ref{pro:equilibria},
the equilibria $\cE$ of the evolution equation \eqref {Abstract-Navier-Stokes2}
correspond to the Killing fields of $\Sigma$.

It is then an interesting question to know how many Killing fields a given manifold can support,
or to be more precise, what the dimension of the vector space of all Killing fields of $\Sigma$ is. (In fact, it turns out that the Killing fields on a Riemannian manifold form a sub Lie-algebra
of the Lie-algebra of all tangential fields).

It might also be worthwhile to recall that the Killing fields of a Riemannian manifold $(M,g)$ are the infinitesimal generators of  the isometries $I(M,g)$ on $(M,g)$, that is, the generators of flows that are isometries on $(M,g)$.
Moreover, in case $(M,g)$ is complete, the Lie-algebra of Killings fields is isometric to 
the Lie-algebra of $I(M,g)$, see for instance Corollary III.6.3 in \cite{Sak96}.

It then follows from  \cite[Proposition III.6.5]{Sak96} that
${\rm dim}\,\cE \le d(d+1)/2,$
where $d$ is the dimension of $\Sigma$.
For compact manifolds, equality holds if and only if $\Sigma$ is isometric to $\mathbb S^d$,
the standard $d$-dimensional Euclidean sphere in $\R^{d+1}$.

On the other hand, if $(M,g)$ is compact and the Ricci tensor is negative definite everywhere,
then any Killing field on $M$ is equal to  zero and $I(M,g)$ is a finite group, see 
\cite[Proposition III.6.6]{Sak96}.
In particular, if $(M,g)$ is a two-dimensional Riemannian manifold with negative Gaussian curvature
then any Killing field is $0$.

\begin{example}\mbox{}
(a) 
Let $\Sigma =\mathbb S^2$. Then ${\rm dim}\,\cE=3$
and each equilibrium $u_* \in \cE$ corresponds to a rotation about an axis spanned by
a vector $\upomega=(\upomega_1,\upomega_2,\upomega_3)\in\R^3$.
Therefore, $u_*\in \cE$ is given by
 $$u_*(x)=\upomega\times x,\quad x\in {\mathbb S}^2,$$
 for some $\upomega\in\R^3$.
 According to Theorem~\ref{thm:stability}, each equilibrium $u_*=\upomega\times x $ is stable and each solution
 $u$ of \eqref {Abstract-Navier-Stokes2} that starts out close to $u_*$ converges at an exponential
 rate towards a (possibly different) equilibrium $u_\infty =\upomega_\infty \times x$ for some $\upomega_\infty\in \R^3$.
 \medskip\\
 (b)
Suppose $\Sigma={\mathbb T}^2$, say with parameterization
 \begin{equation}
 \begin{aligned}
 x_1 &=(R+r\cos\phi)\cos\theta \\
 x_2 &=(R+r\cos\phi)\sin \theta \\
 x_3 &=r\sin \phi,
 \end{aligned}
 \end{equation}
 where $\phi,\theta \in [0,2\pi)$ and $0<r<R$. Then one readily verifies that the velocity field
 $u_*=\omega e_3\times x$, with $\omega\in\R$, is an equilibrium. 
 Hence, the fluid on the torus rotates about the $x_3$-axis with angular velocity $\omega$. According to Theorem~\ref{thm:stability}, all of these equilibria are stable.
 
 \smallskip
 
With the above parameterization, one shows that 
$$K=\frac{\cos \phi}{r(R+r\cos \phi)},$$
where $K$ is the Gauss curvature.
By Gauss' Theorema Egregium, $K$ is invariant under (local) isometries, and this implies that
rotations around the $x_3$-axis are the only isometries on $\mathbb{T}^2$. 
 

 \end{example}

\appendix
\section{}
\noindent

In this appendix we collect some results from differential geometry that are employed throughout the manuscript.
We also refer to \cite[Chapter 2]{PrSi16} for complementary information.

\noindent
We will assume throughout that $\Sigma$ is a smooth, compact, closed (that is, without boundary) hypersurface embedded in $\R^{d+1}$.
We mention on the go that these assumptions imply that $\Sigma$ is orientable, see for instance~\cite{Sam68}.

Let $\nu_\Sigma$ be the unit normal field of $\Sigma$ (which is compatible with the chosen orientation).
Then the orthogonal projection $\cP_\Sigma$ onto the tangent bundle of $\Sigma$ is defined by
$\cP_\Sigma = I-\nu_\Sigma\otimes\nu_\Sigma$.

We use the notation
$\{\tau_1(p), \cdots,\tau_d(p)\}$ to denote a local basis of the tangent space ${\sf T}_p\Sigma$ of $\Sigma$ at $p$, and
$\{\tau^1(p), \cdots,\tau^d(p)\}$ to denote the corresponding dual basis of the cotangent space ${\sf T}^*_p\Sigma$ at $p$.
Hence we have $(\tau^i(p) | \tau_j(p))=\delta^i_j$, the Kronecker delta function.
Note that
\begin{equation*}
\{\tau_1,\cdots,\tau_d\}=\left\{\frac{\partial}{\partial x^1},\cdots,\frac{\partial}{\partial x^d}\right\},
\quad
\{\tau^1,\cdots,\tau^d\}=\left\{dx^1,\cdots, dx^d\right\}.
\end{equation*}
In this manuscript we will occasionally not distinguish between vector fields and covector fields, that is, we identify
\begin{equation}\label{identify}
u=u^i \tau_i = u_i\tau^i,
\end{equation}
where, as usual, the Einstein summation convention is employed throughout.
The metric tensor is given by $g_{ij}=(\tau_i |\tau_j),$ where $(\cdot | \cdot)$ is the Euclidean inner product of $\R^{d+1}$,
and the dual metric $g^*$ on the cotangent bundle ${\sf T}^*\Sigma$ is given by $g^{ij}=(\tau^i|\tau^j)$. It holds that
\begin{equation}\label{metric}
g^{ik}g_{kj}=\delta^i_j,\quad \tau^i=g^{ij}\tau_j,\quad \tau_j=g_{jk}\tau^k.
\end{equation}
Hence, $g$ is induced by the inner product $(\cdot | \cdot)$, that is, we have
\begin{equation}\label{dual-metric}
g(u,v)=g_{ij}u^iv^j= (u^i\tau_i | v^j\tau_j)=(u|v),\quad u=u^i\tau_i,\;\; v=v^j\tau_j,
\end{equation}
whereas
\begin{equation}
g^*(u,v)=g^{ij}u_iv_j=(u_i\tau^i | v_j\tau^j)=(u|v),\quad u=u_i\tau^i,\;\; v=v_j\tau^j.
\end{equation}
It holds that
\begin{equation}\label{Christoffel}
\partial_i \tau_j = \Lambda^k_{ij}\tau_k +l_{ij}\nu_\Sigma, \quad
\partial_i\tau^j = -\Lambda^j_{ik}\tau^k + l^j_i \nu_\Sigma,
\end{equation}
where $\Lambda^k_{ij}$ are the Christoffel symbols,
$l_{ij}$ are the components of the second fundamental form, and
$l^i_j$ are the components of the Weingarten tensor $L_\Sigma$; that is, we have
\begin{equation}\label{second fundamental}
l_{ij}=(\partial_j\tau_i| \nu_\Sigma)=-(\tau_i | \partial_j\nu_\Sigma), \quad l^j_i=g^{jk}l_{ki},
\end{equation}
and
\begin{equation}\label{Weingarten}
L_\Sigma \tau_j=-\partial_j\nu_\Sigma, \quad
L_\Sigma 
=l^j_i\tau^i\otimes\tau_j.
\end{equation}

If $\varphi\in C^1(\Sigma, \R)$, the surface gradient of $\varphi$ is defined by
$\nabla_\Sigma\varphi=\partial_i\varphi \tau^i$. 
If $u$ is a $C^1$-vector field on $\Sigma$ (not necessarily tangential) we define the surface gradient of $u$ by
\begin{equation}\label{surface-gradient}
\nabla_\Sigma u= \tau^i\otimes \partial_i u.
\end{equation}
It follows from~\eqref{Christoffel} that
\begin{equation}\label{derivative-local}
\partial_i u = \partial_i (u^j\tau_j)=(\partial_i u^j  + \Lambda^j_{ik}u^k)\tau_j+ l_{ij} u^j \nu_\Sigma
\end{equation}
for a tangential vector field $u$.
The \emph{covariant derivative} $\nabla_i u$ of a tangential vector field is defined by
$\nabla_i u =P_\Sigma \partial_i u.$
Hence we have
\begin{equation}\label{covariant-local}
\nabla_i u=(\partial_i u^j  + \Lambda^j_{ik}u^k)\tau_j\quad\text{for}\quad u= u^j\tau_j.
\end{equation}

The surface divergence ${\rm div}_\Sigma u$ for a (not necessarily tangential) vector field $u$ is defined by
\begin{equation}\label{divergence-definition}
{\rm div}_\Sigma u = (\tau^i | \partial_i u).
\end{equation}
Then the $d$-fold mean curvature $\kappa_\Sigma$  of $\Sigma$ is given by
\begin{equation}\label{mean-curvature}
\kappa_\Sigma =-{\rm div}_\Sigma \nu_\Sigma
= -(\tau^i | \partial_i \nu_\Sigma)=(\tau^i | L_\Sigma\tau_i)
={\rm tr}L_\Sigma=l^i_i.
 \end{equation}
Hence, $\kappa_\Sigma$ is the trace of $L_\Sigma$ (which equals the sum of the principal curvatures).
For a vector field $u=v^j\tau_j+w\nu_\Sigma$, it follows
from~\eqref{derivative-local}, \eqref{mean-curvature} and the fact that $\nu_\Sigma$ and $\tau^i$ are orthogonal
\begin{equation}\label{divergence-local}
{\rm div}_\Sigma u =(\tau^i |\partial_i u)=(\partial_i v^i  +\Lambda^i_{ik}v^k) -w\kappa_\Sigma.
\end{equation}
For a tangent vector field $u$ and a scalar function $\varphi$, the surface divergence theorem states that
\begin{equation}\label{divergence-theorem}
\int_\Sigma (\nabla_\Sigma \varphi | u)\,d\Sigma = -\int_\Sigma \varphi\, {\rm div}_\Sigma\, u\,d\Sigma.
\end{equation}
For a tensor $K=k^j_i \tau^i\otimes\tau_j$, the surface divergence is defined by
\begin{equation}\label{divergence-tensor}
{\rm div}_\Sigma K = (\partial_i K)^{\sf T}\tau^i.
\end{equation}
Hence,
\begin{equation}\label{divergence-P}
{\rm div}_\Sigma \cP_\Sigma = \partial_i (\tau_j\otimes\tau^j)\tau^i = \kappa_\Sigma \nu_\Sigma.
\end{equation}
\goodbreak
\begin{lemma}\label{lem:app}
Suppose $\varphi$ is a $C^1$-scalar function and $u,v,w$ are $C^1$-tangential vector fields on $\Sigma$. Then
\begin{itemize}
\item[(a)] ${\rm div}_\Sigma (\varphi \cP_\Sigma) = \nabla_\Sigma \varphi + \varphi \kappa_\Sigma \nu_\Sigma$.
\vspace{2mm}
\item[(b)] $ (u\cdot\nabla_\Sigma v) = (\nabla_\Sigma v)^{\sf T} u$.
\vspace{2mm}
\item[(c)] $\nabla_\Sigma (u|v)=  (\nabla_\Sigma u)v + (\nabla_\Sigma v)u$.
\vspace{2mm}
\item[(d)] $\big(u \big |\nabla_\Sigma (v|w)\big)
=\big(u\cdot\nabla_\Sigma v \big| w \big) + \big(u\cdot\nabla_\Sigma w \big| v \big) .$
\vspace{2mm}
\end{itemize}
\end{lemma}
\begin{proof}
(a)
It follows from \eqref{divergence-tensor} that
\begin{equation*}
{\rm div}_\Sigma (\varphi\cP_\Sigma) =\partial_i (\varphi\cP_\Sigma)\tau^i
=(\partial_i\varphi) \cP_\Sigma \tau^i +\varphi \partial_i (\cP_\Sigma)\tau^i
= \partial_i\varphi \tau^i +\varphi {\rm div}_\Sigma \cP_\Sigma.
\end{equation*}
The assertion is now a consequence of \eqref{divergence-P}.
\smallskip\\
\noindent
{ (b)} Using local coordinates, we obtain
\begin{equation*}
u\cdot\nabla_\Sigma v
= u^i\partial_i v= (\partial_i v\otimes \tau^i)u=(\nabla_\Sigma v)^{\sf T}u.
\end{equation*}
 (c) In local coordinates, $\partial_i (u|v) =(\partial_i u | v) +(u | \partial_i v)$.
It is now easy to conclude that
\begin{equation*}
 \nabla_\Sigma  (u|v) = \partial_i (u|v) \tau^i
= (\tau^i \otimes\partial_i u) v  + (\tau^i \otimes \partial_i v)u = (\nabla_\Sigma u)v + (\nabla_\Sigma v)u.
\end{equation*}
(d)
This follows from the assertions in (b) and (c).
\end{proof}

Let
$$
\cD_\Sigma (u):=\frac{1}{2}\cP_\Sigma \big( \nabla_\Sigma u +[\nabla_\Sigma u]^{\sf T}\big)\cP_\Sigma.
$$
Suppose $u\in C^2(\Sigma, {\sf T}\Sigma)$, $v\in C^1(\Sigma, {\sf T}\Sigma)$.
Then one shows that
\begin{equation}\label{divergence-parts}
\begin{aligned}
({\rm div}_\Sigma 
\cD_\Sigma(u)|v) &= {\rm div}_\Sigma (\cD_\Sigma(u)v )- \cD_\Sigma(u): \nabla_\Sigma v \\
({\rm div}_\Sigma 
\cD_\Sigma(u)|u) & = {\rm div}_\Sigma (\cD_\Sigma(u)u) - |\cD_\Sigma(u)|^2 \\
{\rm tr}\, \cD_\Sigma(u)&={\rm div}_\Sigma u,
\end{aligned}
\end{equation}
where $\cD_\Sigma(u):\nabla_\Sigma v=(\cD_\Sigma(u) \tau^j | (\nabla_\Sigma v)^{\sf T} \tau_j)$, 
$|\cD_\Sigma(u)|^2=\cD_\Sigma(u):\cD_\Sigma(u)$, 
and 
${\rm tr}\,\cD_\Sigma (u)= (\cD_\Sigma (u) \tau^j | \tau_j)$.

\medskip
\begin{proposition}\label{pro:D-Sigma}
Suppose $u\in C^2(\Sigma,{\sf T}\Sigma)$.
\begin{itemize}
\item[(a)]
Then we have the following representation:
\begin{equation*}
2\cP_\Sigma \,{\rm div}_\Sigma\, \cD_\Sigma(u)
=\Delta_\Sigma u+ \nabla_\Sigma\, {\rm div}_\Sigma u + (\kappa_\Sigma L_\Sigma - L^2_\Sigma )u,
\end{equation*}
where $\Delta_\Sigma$ is the Bochner-Laplacian (also called the conformal Laplacian),
defined in local coordinates by
\begin{equation*}
\Delta_\Sigma =g^{ij}(\nabla_i\nabla_j - \Lambda^k_{ij}\nabla_k).
\end{equation*}
\item[(b)]
It holds that
\begin{equation*}
\label{Ricci-Gauss-2}
\begin{aligned}
& (\kappa_\Sigma L_\Sigma - L^2_\Sigma) u = {\rm Ric}_\Sigma u \\
& (\kappa_\Sigma L_\Sigma - L^2_\Sigma) u = K_\Sigma u\quad \text{in case $d=2$,}
\end{aligned}
\end{equation*}
where  ${\rm Ric}_\Sigma$ is the Ricci tensor and
 $K_\Sigma$  the Gaussian curvature of $\Sigma$,
 respectively. \\
\end{itemize}
\end{proposition}

\begin{proof}
(a)
We note that in local coordinates, $\cD_\Sigma(u)$ is given by
\begin{equation}\label{D-Sigma=}
\begin{aligned}
2\cD_\Sigma(u)
&= \tau^i\otimes P_\Sigma \partial_i u + P_\Sigma\partial_i u\otimes \tau^i
= \tau^i\otimes \nabla_i u + \nabla_i u\otimes \tau^i.
\end{aligned}
\end{equation}
From \eqref{divergence-tensor} and  the relation $\cP_\Sigma = I-\nu_\Sigma\otimes\nu_\Sigma$ follows
\begin{equation*}
\begin{aligned}
      \cP_\Sigma {\rm div}_\Sigma  ( \nabla_i u\otimes \tau^i )
&= \cP_\Sigma \partial_j ( \tau^i \otimes \cP_\Sigma\partial_iu ) \tau^j \\
&= (\cP_\Sigma \partial_j\tau^i \otimes \cP_\Sigma\partial_i u)\tau^j +
        (\tau^i\otimes \partial_j \cP_\Sigma\partial_i u)\tau^j\\
&= (\partial_i u | \tau^j)\cP_\Sigma \partial_j\tau^i
       + \tau^i\big(\partial_j(\partial_i u- [\nu_\Sigma\otimes \nu_\Sigma]\partial_i u) | \tau^j\big) \\
& =  (\partial_i u | \tau^j)\cP_\Sigma \partial_j\tau^i + \tau^i (\partial_j\partial_i u|\tau^j)
-   \tau^i(\partial_j\nu_\Sigma | \tau^j)(\nu_\Sigma |\partial_i u),
\end{aligned}
\end{equation*}
where, in the last line, we employed the relation $(\nu_\Sigma | \tau^j)=0$. 
Next, we observe that
\begin{equation*}
\begin{aligned}
(\partial_j\partial_i u|\tau^j)\tau^i
& = \partial_i (\partial_j u| \tau^j) \tau^i - (\partial_j u | \partial_i \tau^j)\tau^i \\
& = \nabla_\Sigma {\rm div}_\Sigma u - (\partial_j u| -\Lambda^j_{ik}\tau^k +l^j_i\nu_\Sigma)\tau^i \\
& =  \nabla_\Sigma {\rm div}_\Sigma u -(\partial_j u| \tau^k) \cP_\Sigma \partial_k\tau^j  - l^j_i (\partial_j u | \nu_\Sigma)\tau^i \\
& =  \nabla_\Sigma {\rm div}_\Sigma u -(\partial_j u| \tau^k) \cP_\Sigma \partial_k\tau^j  - L^2_\Sigma u,
\end{aligned}
\end{equation*}
where we used \eqref{derivative-local} and  the relations
$L_\Sigma \tau^m = l^m_r\tau^r$ as well as $L_\Sigma \tau_m=l_{m r}\tau^r$
to deduce
\begin{equation}\label{R1}
l^j_i(\partial_j u | \nu_\Sigma)\tau^i =l^j_i l_{jk} u^k \tau^i = L_\Sigma  (u^k l_{jk}\tau^j )
= L^2_\Sigma u_k\tau^k =L^2_\Sigma u.
\end{equation}
From  \eqref{derivative-local} and \eqref{mean-curvature} follows
\begin{equation}\label{R2}
- (\partial_j\nu_\Sigma | \tau^j)(\nu_\Sigma |\partial_i u)\tau^i = l^j_j l_{ik} u^k \tau^i
= \kappa_\Sigma L_\Sigma u^k\tau_k = \kappa_\Sigma L_\Sigma u.
\end{equation}
Summarizing we have shown that
\begin{equation*}
  \cP_\Sigma {\rm div}_\Sigma  ( \nabla_i u\otimes \tau^i )
 = \nabla_\Sigma {\rm div}_\Sigma u + \kappa_\Sigma L_\Sigma u  -L^2_\Sigma u .
\end{equation*}
Moreover, we have
\begin{equation*}
\begin{aligned}
      \cP_\Sigma {\rm div}_\Sigma  (\tau^i \otimes \nabla_iu)
&= \cP_\Sigma \partial_j (\nabla_iu  \otimes \tau^i)\tau^j \\
&= g^{ij}\cP_\Sigma \partial_j \nabla_iu  + (\partial_j \tau^i|\tau^j) \nabla_i u  \\
& =  g^{ij}(\nabla_i \nabla_j u -\Lambda^k_{ij}\nabla_k u) =\Delta_\Sigma u.
\end{aligned}
\end{equation*}

\medskip\noindent
(b)
The computations in \eqref{R1}-\eqref{R2}  show that  in local coordinates
\begin{equation*}
(\kappa_\Sigma L_\Sigma -L^2_\Sigma)u=g^{jm}(l_{jm}l_{ik}- l_{im} l_{jk})u^k \tau^i.
\end{equation*}
By the Gauss equation, see for instance \cite[Proposition II.3.8]{Sak96}, this yields
\begin{equation*}
\begin{aligned}
(\kappa_\Sigma L_\Sigma -L^2_\Sigma)u
&=g^{jm}R_{jikm} u^k\tau^i
=R_{ik} u^k\tau^i
= {\rm Ric}_\Sigma u,
\end{aligned}
\end{equation*}
where $R_{jikm}$ are the components of the curvature tensor and
$R_{ik}$ the components of the Ricci $(0,2)$-tensor.
In case that $\Sigma$ is a surface embedded in $\R^3$, one obtains
\begin{equation*}
(\kappa_\Sigma L_\Sigma -L^2_\Sigma)u= K_\Sigma u,
\end{equation*}
where $K_\Sigma$ is the Gauss curvature of $\Sigma$.
This can, for instance, be seen as follows:
\begin{equation*}
\begin{aligned}
(\kappa_\Sigma L_\Sigma -L^2_\Sigma)u
&=g^{jm}(l_{jm}l_{ik}- l_{im} l_{jk})u^k \tau^i \\
&= (l^j_j l^k_i-l^j_i l^k_j)u_k\tau^i = \det (L_\Sigma)\, \delta^k_i u_k\tau^i = K_\Sigma u.
\end{aligned}
\end{equation*}
The proof of Proposition~\ref{pro:D-Sigma} is now complete.
\end{proof}
{
\begin{remarks}\label{rem:app}
{ (a)}
We note that in local coordinates,
\begin{equation*}
\begin{aligned}
2\cD_\Sigma(u)
=(\partial_i u_j -\Lambda^k_{ij}u_k)\tau^i\otimes \tau^j
+ (\partial_i u_j -\Lambda^k_{ij}u_k)\tau^j\otimes \tau^i
\end{aligned}
\end{equation*}
for $u=u_j\tau^j\in C^1(\Sigma, {\sf T}^*\Sigma)$ and
\begin{equation*}
\begin{aligned}
2\cD_\Sigma(u)
=(\partial_i u^j  +\Lambda^j_{ik}u^k)\tau^i\otimes \tau_j
+ (\partial_i u^j +\Lambda^j_{ik}u^k)\tau_j\otimes \tau^i
\end{aligned}
\end{equation*}
for $u=u^j\tau_j\in C^1(\Sigma, {\sf T}\Sigma)$.
\medskip\\
\noindent
{ (b)}
Suppose $u,v$ are tangential fields on $\Sigma$.  Then
\begin{equation}\label{connection}
\nabla_v u := (\nabla_i u\otimes \tau^i)v
\end{equation}
coincides with the Levi-Civita connection $\nabla$ of $\Sigma$. \\
We note that $\nabla_i =\nabla_{\tau_i}=\nabla_{\frac{\partial \;}{\partial x^i}}$
and $\nabla u= \nabla_i u\otimes \tau^i$ in local coordinates.
\medskip\\
\noindent
(c)
A straightforward computation shows that in local coordinates
\begin{equation*}
g^{ij}(\nabla_i\nabla_j u - \Lambda^k_{ij}\nabla_ku)
=g^{ij}(\nabla^2 u)(\tau_i,\tau_j),\quad u\in C^2(\Sigma, {\sf T}\Sigma).
\end{equation*}
Hence,  $\Delta_\Sigma ={\rm tr}_g\,(\nabla^2 u)$.
\medskip\\
\noindent
{ (d)}
Let $\nabla $ be the Levi-Civita connection of $\Sigma$.
Then it follows from (b) and~\eqref{D-Sigma=}
\begin{equation*}
\cD_\Sigma (u)= \frac{1}{2}\left( \nabla u + [\nabla u]^{\sf T}\right).
\end{equation*}
\noindent
{(e)} Suppose $u,v,w$ are $C^1$-tangential fields.
Employing  \eqref{D-Sigma=} and \eqref{connection}, one readily verifies that
\begin{equation*}
(\cD_\Sigma(u)v | w) + (\cD_\Sigma (u)w|v)= (\nabla_v u | w) + (\nabla_w u | v).
\end{equation*}
We remind that a tangential field $u$ on $\Sigma$ is called a {\em Killing field} if
\begin{equation*}
(\nabla_v u |w) + (\nabla_w u|v)=0\quad \text{for all tangential fields $v,w$ on $\Sigma$},
\end{equation*}
see for instance \cite[Lemma III.6.1]{Sak96}.
This implies for a $C^1$-tangent field $u$
\begin{equation*}
\cD_\Sigma(u)=0\;\; \Longleftrightarrow \;\; \text{$u$ is a Killing  field}.
\end{equation*}
\end{remarks}
}

\bigskip
{\bf Acknowledgments:}
This is the last joint work with our close friend Jan Pr\"uss, who passed away before this manuscript was completed. 
We are grateful for his input and ideas which are now incorporated in this manuscript.  He is deeply missed.

\smallskip
Lastly, we would also like to express our thanks to Marcelo Disconzi for helpful discussions.


\end{document}